\documentclass[11pt,a4paper]{article}
\usepackage{amsmath, amssymb, graphics, setspace}
\usepackage{empheq}
\usepackage{amssymb,amsbsy,amsmath,amsfonts,amssymb,amscd,amsthm, mathrsfs, bbm, bigints}
\usepackage{amssymb}
\usepackage{color}
\usepackage{hyperref}

\newcommand{\1}{\mathbbm{1}}
\newcommand{\CC}{\mathbb{C}}

\newcommand{\RR}{\mathbb{R}}

\numberwithin{equation}{section}
\newtheorem{theo}{Theorem}[section]
\newtheorem{prop}[theo]{Proposition}
\newtheorem{lem}[theo]{Lemma}
\newtheorem{cor}[theo]{Corollary}
\newtheorem{rem}[theo]{Remark}

\newcommand{\beqn}{\begin{equation}}
\newcommand{\eeqn}{\end{equation}}
\newcommand{\bear}{\begin{eqnarray}}
\newcommand{\eear}{\end{eqnarray}}
\newcommand{\bean}{\begin{eqnarray*}}
\newcommand{\eean}{\end{eqnarray*}}

\newcommand{\dd}{\, {\rm d}}

\begin{document}
\title{Instability of singular equilibria of a wave kinetic equation}
\maketitle
\begin{center}
{ Miguel Escobedo}\\
{\small Departamento de Matem\'aticas,} \\
{\small Universidad del
Pa{\'\i}s Vasco,} \\
{\small Apartado 644, E--48080 Bilbao, Spain.}\\
{\small E-mail~: {\tt miguel.escobedo@ehu.es}}
\end{center}

\begin{center}
{Angeliki Menegaki} \\
{\small Department of Mathematics,} \\
{\small Imperial College London,}\\
{\small Exhibition Road, London, SW7 2AZ, UK}\\
{\small E-mail~: {\tt a.menegaki@imperial.ac.uk}}
\end{center}
\bigskip

\noindent
{\bf Abstract}: We consider the singular Rayleigh-Jeans equilibrium of the $4$-waves kinetic turbulence equation for the three dimensional Schr\"odinger equation. We first show the formation in finite time of a Dirac measure at zero frequency in the solution of the  wave kinetic equation when the initial data  has the form of Rayleigh-Jeans, truncated at large values of the energy. The initial value problem for the linearization around the singular Rayleigh Jeans equilibria is then solved in several functional spaces. Then, long time convergence to a Dirac measure at the origin is described in detail for  some of the solutions. This determines a basin of attraction of the Dirac measure.


\bigskip
\section{Introduction}

The description, within the framework of wave turbulence theory, of large systems of weakly interacting waves gives rise to the study of nonlinear kinetic equations, sometimes referred to as kinetic wave equations (KWE).  These equations are considered to describe the evolution of the density function $n(t, \bf p)$ of waves with wave number $\bf p$ at time $t$. These equations have a family of equilibrium solutions called  Rayleigh Jeans (RJ) thermodynamic equilibria. They can also have other stationary, non-equilibrium solutions known as Kolmogorov-Zakharov (KZ) spectra that cancel the nonlinear collision integral without the integrand being identically zero. These spectra correspond to non-zero wave action and/or energy fluxes (cf. \cite{Zbook}, \cite{DNPZ}, \cite{N}). 

General stability properties of these stationary solutions may be found  in \cite{Zbook} (Chapter 4) or \cite{BZ} (Chapter 3), although mainly for the KZ spectra, with only some short remarks about the RJ equilibria (Section 4.2.3 of \cite{Zbook} or  Section 4.1 of \cite{BZ}). 

\subsection{The ``Schrödinger WKE''.}

In what follows, we only consider the WKE for a system of waves governed by the nonlinear Schrödinger equation in $\RR^3$, which will sometimes be referred to as Schrödinger WKE. This is the only case where the kinetic equation has been rigorously deduced from the corresponding dispersive equation for waves, but only for a spatially homogeneous and isotropic system. First  at equilibrium in \cite{LukSpohn}, and  more recently at non equilibrium, to arbitrarily large time scales, up to the time where the kinetic equation is well-posed, in \cite{DH3}, (cf. also  \cite{BGHS, CollGermain, CollGermain2, DH1} for previous results). No such deduction is known for a  non isotropic system of waves, although the corresponding WKE has been widely  considered (cf. \cite{Zbook, DNPZ}).

The $L^2$ stability property of RJ equilibria is proved in \cite{Men23} for the spatially homogeneous Schrödinger WKE  with frequency cut off. There the linearized operator around a RJ equilibria with $\mu \le 0$ is shown to be coercive in $L^2$. Local well-posedness in weighted $L^2$ and $L^\infty$  for more general initial data for wave kinetic equations similar to the Schrödinger WKE without cut off and  with slightly more general dispersion relations is proved in \cite{GIT20}. 

On the other hand, for the spatially inhomogeneous WKE, global in time existence, uniqueness and stability of mild solutions for small initial data that depend on the space variable and not necessarily isotropic in frequency are proved in \cite{A}. Their result shows that the equilibrium $n\equiv 0$ is stable under sufficiently small perturbations with uniform algebraic  decay  at infinity both in space and frequencies.  

Let us now consider the Schr\"odinger WKE and examine the stability of RJ equilibria in the simplified situation of spatially homogeneous and radially symmetric distributions $n(t, p)$, where $p=|{\bf p}|$. With a slight abuse of notation, we write $n({\bf p})=n(p)$). The equation in this case reads 
\begin{align} \label{eq:NL WKEB}
&\partial_t n(t,p) = \iint_{D(p)} \frac {\min\{p, p_1, p_2, p_3\}} {p } \left(n_3n_4(n_1+n_2)-n_1n_2(n_3+n_4)\right) p_2p_3 \dd p_2\dd p_3\\
&D(p)=\left\{(p_2, p_3);\,\,p_2>0, p_3>0, p_2+p_3>p \right\},\,\,p_1=p_2+p_3-p.\nonumber
\end{align}
Or also, in terms of the new variables $\omega =p^2$, $f(\omega )=n(\bf p)$,
\begin{align}
&\partial_{t}f=Q(f(t))\label{eq:NL WKEC10}\\
&Q(f(t))=\iint _{ D(\omega _1) }  W\left[  \left(  f_{1}+f_{2}\right)  f_{3}
f_{4}-\left(  f_{3}+f_{4}\right)  f_{1}f_{2}\right]  \dd \omega_{3} \dd \omega
_{4},\ \ t>0 \label{eq:NL WKEC11}\\
&W=\frac{\min\left\{  \sqrt{\omega_{1}},\sqrt{\omega_{2}},\sqrt{\omega_{3}
},\sqrt{\omega_{4}}\right\}  }{\sqrt{\omega_{1}}}\ \ ,\ \ \ \omega_{2}
=\omega_{3}+\omega_{4}-\omega_{1}. \label{eq:NL WKEC13} 
\end{align}
Due to the change to radial variables, the functions $n$ and $f$ do not describe  any more a density of waves  and the conservation of action wave and energy are respectively:
\begin{align}
&\int _0^\infty\!\!\! n(t, p)p^2\dd p=\int _0^\infty \!\!\!n(0, p)p^2\dd p,\,\,\,\int _0^\infty \!\!\!f(t, \omega )\omega^{1/2}\dd \omega=\int _0^\infty\!\!\! f(0, \omega )\omega^{1/2}\dd \omega,\,\,\forall t>0 \label{CnsM}\\
&\int _0^\infty \!\!\!n(t, p)p^4\dd p=\int _0^\infty\!\!\! n(t, p)p^4\dd p,\,\,\,\int _0^\infty \!\!\!f(t, \omega )\omega^{3/2}\dd\omega=\int _0^\infty \!\!\! f(0, \omega )\omega^{3/2}\dd\omega,\,\forall t>0.\label{CnsE}
\end{align}
The KZ spectra are $Q^{1/3}p^{-3}$ and $P^{1/3}p^{-7/3}$ with $Q$ and $P$ real constants  corresponding to wave action and energy fluxes respectively (cf. \cite{Zbook}, \cite{DNPZ}). 
The \ radially symmetric RJ equilibria of (\ref{eq:NL WKEB}) or (\ref{eq:NL WKEC10}) are 
\begin{align}
n(p)= \frac {1} {\omega (p)-\mu },\,\,\mu \le 0\label{ERJ1}
\end{align}
for which  the integrand is identically zero. In all that follows only the case of (\ref{ERJ1}) for $\mu =0$, that we sometimes call singular RJ, is considered because our arguments based on homogeneity do not apply when $\mu <0$.

\subsection{Set up and Main results}
Our first result shows that  singular RJ  distributions are  unstable equilibria for the non linear equation (\ref{eq:NL WKEC10})  under some very simple perturbations. For that purpose the following definition of weak solution of the nonlinear equation is used. Consider the sub space of non-negative measures on $[0, \infty)$:
\begin{align*}
&\mathcal{M}^\rho _{+}\left(\left[  0,\infty\right)\right) =\left\{\mu \in  \mathcal{M}_{+}([0, \infty); \| \mu \| _{ \rho  }<\infty\right\},\,\,\rho \in \RR,\\
&\text{where},\,\|\mu \| _{ \rho  }=\sup _{ R>1 } \frac {1} {(1+R)^\rho }\frac{1}{R}\int_{\frac{R}{2}}^{R}\mu\left(  \dd \omega\right)
+\int _0^1 \mu (\dd \omega ).
\end{align*}
and let  $f_0$ be a non negative measure  such that $g_0$, the measure defined by $g_0(\omega )=\sqrt \omega f_0(\omega )$, satisfies  $g_0\in \mathcal{M}^\rho _{+}\left(\left[  0,\infty\right)\right) $ for some $\rho <-1$. 

We say that a family of  measures $f(t, \cdot)$ for $t>0$ is a weak solution of  (\ref{eq:NL WKEC10}) with initial data $f_0$ if $g(t, \cdot)$  defined as $g(t, \omega )=\sqrt{\omega }f(t, \omega )$ satisfies:\\
1.- $g \in C([0, \infty); \mathcal{M}^\rho _{+}\left(\left[  0,\infty\right)\right) )$\\
2.-for all  $\varphi\in C_{0}^{2}\left(  \left[  0,\infty \right)  \times\left[  0,\infty\right)  \right)$ and all $t^*>0$:
\begin{align}
 & \int_{\left[  0,\infty\right)  }g\left(  t_{\ast},\omega\right)
\varphi\left(  t_{\ast},\omega\right)  \dd \omega  -\int_{\left[  0,\infty\right)
}g_{0}\varphi\left(  0,\omega\right)  \dd \omega
  =\int_{0}^{t_{\ast}}\int_{\left[  0,\infty\right)  }g\partial_{t}\varphi
\dd \omega \dd t  \nonumber \\
&  +\int_{0}^{t_{\ast}}\iiint_{\left[  0,\infty\right)^3 }\frac{g_{1}g_{2}g_{3}
\Phi}{\sqrt{\omega_{1}\omega_{2}\omega_{3}}}\times \nonumber \\
&\hskip 2cm  \times\left[  \varphi\left(
\omega_{1}+\omega_{2}-\omega_{3}\right)  +\varphi\left(  \omega_{3}\right)
-\varphi\left(  \omega_{1}\right)  -\varphi\left(  \omega_{2}\right)  \right]
\dd \omega_{1} \dd \omega_{2} \dd \omega_{3} \dd t  \label{Z2E1}
\end{align}
where $\Phi =\min \{ \sqrt{\omega _1}, \sqrt{\omega _2}, \sqrt{\omega _3}, \sqrt{\omega _4} \}$,\\
3.- for all $t>0$,
\begin{align*}
&\int  _{ [0, \infty) }g(t, \omega ) \dd \omega =\int  _{ [0, \infty) }g_0(\omega ) \dd \omega,\\
&\int  _{ [0, \infty) }g(t, \omega ) \omega  \dd \omega =\int  _{ [0, \infty) }g_0(\omega ) \omega \dd \omega
\end{align*}
and we say that $f$ conserves the mass and the energy.

It was proved in Theorem 2.19 and Proposition 2.28 of  \cite{EV3} that for all initial data $f_0$ as above there exists at least one non negative weak solution $f$ in the sense of the above definition. Our first result is then,

\begin{theo}
\label{Thm}
There exists a numerical constant $C>0$ such that, for all $A>0$, $B>0$, $\rho <-1$, if  $g_0\in \mathcal{M}^\rho _{+}\left(\left[  0,\infty\right)\right) $,  
\begin{equation*}
f_0(\omega )=A\omega ^{-1}\1 _{ 0<\omega <B }+\min \left(A\omega ^{-1}, \frac {g_0(\omega )} {\sqrt \omega }\right)\1 _{ \omega >B },
\end{equation*}
and $f$ is one weak non negative solution of (\ref{eq:NL WKEC10}), then 
\begin{equation*}
g(t, \{0\})>0,\,\,\forall t>C A^{-2}.
\end{equation*}
where $g$ is defined as $g(t, \omega )=\sqrt \omega f(t, \omega )$.
\end{theo}

\begin{rem} 
\label{rem1}
Theorem \ref{Thm} shows that for initial data that behave like a singular RJ close to the origin, a Dirac's measure is formed after some finite time $t^*=C A^{-2}$ and this  time may be as small as desired, depending on the constant $A$. Local-well posedness of the equation up to time $t_*$ in some space of locally bounded functions is natural to be conjectured and is under current investigation, following techniques in \cite{EscMischVel07, EscMischVel08}. 
\end{rem}

\begin{rem} 
Theorem \ref{Thm} is very natural in view of part b) of Theorem 6.1 in \cite{XL} for the Nordheim equation, and its proof follows along similar arguments, since for non-negative density functions the collision integral of the Nordheim's equation is bounded from below by the collisional operator  of the wave turbulence equation. 
\end{rem}

After  Theorem \ref{Thm}, our motivation to consider  the linearization of the equation (\ref{eq:NL WKEC10}) around the RJ equilibrium $n^0(p)=p ^{-2}$ is twofold.  

The instability of singular RJ equilibrium shown in Theorem \ref{Thm} comes from the concentration of waves at zero frequency to form a Dirac's delta at $\omega (p)=0$ in finite time. Precise information is lacking about how this condensation phenomena happens (cf. for example \cite{LLPR}). A concentration is still present in the linearized equation, although it takes there an infinite time.  Indeed, some of the solutions $A$ of the linear equation converge  to a Dirac measure asymptotically as $t\to \infty$, although $A(t)$ remains a locally bounded function on $(0, \infty)$ for all time.
In the linear equation it is possible to describe  this concentration mechanism in detail, to show in particular that it happens with an explicit exponential rate and to describe the pointwise  behavior of some of the solutions.

On the other hand, the linearized equation around stationary solutions  have proved to be useful in order to solve initial value problems for the nonlinear equation with a prescribed behavior as $p\to 0$. That is the case still for equation (\ref{eq:NL WKEC10}) and its linearization around the KZ spectrum $\omega ^{-7/6}$  as used in \cite{EscMischVel07, EscMischVel08}. It could then be useful in order to obtain local-well posedness for  the equation (\ref{eq:NL WKEB}) or (\ref{eq:NL WKEC10}).

\subsubsection{Linearized equation} 

The  linearization of equation (\ref{eq:NL WKEB})  around the   RJ  equilibrium $n^0(p)=|p|^{-2}$ has been considered in \cite{MM}, where the perturbation is written as follows,
\begin{align}
\label{EL.1}
n(t, p)=n^0(p)(1+A(t, p)).
\end{align}   
If $q(n)=n_3n_4(n_1+n_2)-n_1n_2(n_3+n_4)$, a simple computation yields,
\begin{align*}
q(n^0(1+A))=A_1\left(n^0_3n^0_4-n^0_2n^0_3-n^0_2n^0_4 \right)-A_2 n^0_3n^0_4+A_3 n^0_2n^0_4+A_4 n^0_2n^0_3
\end{align*}
where $n^{0}_2$ denotes $ n^{0}(p_4+p_3-p_1)$. If  the perturbation $A$ is such that 
\begin{align*}
\iint  _{ D(p_1) }\frac {\min\{p_1, p_2, p_3, p_4\}} {p_1}A(p_3) n^{0}(p_4+p_3-p_1)n^0(p_4)\dd p_3 \dd p_4<\infty, 
\end{align*}
then by Fubini's Theorem,
\begin{align}
\iint _{ D(p_1) }&\frac {\min\{p_1, p_2, p_3, p_4\}} {p_1}q(n^0(1+A_1)) \dd p_3 \dd p_4=\nonumber\\
=&A_1\iint  _{ D(p_1) }\frac {\min\{p_1, p_2, p_3, p_4\}} {p_1}\left(n^0_3n^0_4-n^0_2n^0_3-n^0_2n^0_4 \right) \dd p_3 \dd p_4 \nonumber\\
&-\iint  _{ D(p_1) }\frac {\min\{p_1, p_2, p_3, p_4\}} {p_1} A_2n^{0}_3n^{0}_4 \dd p_3 \dd p_4 \nonumber\\
&+2
\iint  _{ D(p_1) }\frac {\min\{p_1, p_2, p_3, p_4\}} {p_1} A_4n^{0}_2n^{0}_3 \dd p_3 \dd p_4. \label{EFF1}
\end{align}

In this case it was shown  in \cite{MM} that the linearized equation for $A$ may be written,
\begin{align}
\label{E1}
\frac {\partial A(t, p)} {\partial t}=\gamma_N  A(t, p)+\int _0^\infty \mathcal V\left(\frac {p} {p'}\right)A(t, p')\frac {\dd p'} {p'} 
\end{align} 
where $\mathcal V$ is an explicit function and $\gamma_N$ is an explicit negative constant, given in the Appendix below.

The findings in \cite{MM} regarding the linear equation \eqref{E1} are described in some detail in the Appendix,  Section \ref{Sapp}, but  can be briefly summarised here as follows. First, the explicit expressions of the function $\mathcal V$ and its Mellin transform $W _{ \mathcal V }$ were obtained in terms of special functions. Solutions of the equation (\ref{E1}) were obtained using the Mellin transform and for some of them, the asymptotic behavior as $p\to 0$ and as $p \to \infty$, was described by means of the steepest descent method. It was also observed that  
for some initial data, all the waves tend to concentrate exponentially fast at the value $p=0$.

Equation (\ref{E1}) is considered here in more strictly mathematical terms, aiming in particular at establishing the well-posedness of the equation, define and estimate the linear semigroup defined by the equation, and determine the basin of attraction of the Dirac measure at the origin. These aspects could not be addressed in detail in \cite{MM} are of interest not only in their own right but also for considering the nonlinear equation.

Our main results on the linearised equation \eqref{E1} are summarised  as follows. Define the   spaces:
\begin{align}
&L_\theta^q\left(0, \infty \right)=\left\{A\, \text{ is measurable on}\, (0, \infty);\,\int _0^\infty |A(p)|^q p^{2\theta-1} \dd p <\infty \right\}
\label{Space1}\\
& \mathscr Y _{ A, B }(0, \infty) := \{g\in L^\infty _{ \text{loc} }(0, \infty);  ||g|| _{A, B }<\infty\},\,\,\text{where},\label{Space2}\\
& ||g|| _{A, B }:=\sup _{ 0<p<1 }(p^{2A}|g(p)|)+\sup_{p>1} (p^{2B}|g(p)|)\nonumber.
\end{align}
 
 \begin{theo}\label{Theo2:Linear eq}  The linear equation \eqref{E1} has the following properties:
\begin{itemize} 
\item \emph{(Well-Posedness)} (i) For all non negative $A_0\in   L_\theta^q\left(0, \infty\right)$ with $q\ge 1$ and $\theta\in (-1, 5/2)$ there exists a unique  $ A \in C([0, \infty); L^q_\theta( 0, \infty) ) \cap 
C^1((0, \infty); L_\theta^q(0, \infty) )$ such that $A(t)\ge 0$ for all $t > 0$ and satisfying equation (\ref{E1}) in $L^q_\theta( 0, \infty) $. Moreover,   for all $T>0$ there exists a constant $C_T>0$ satisfying,
\begin{align*}
\|A(t)\|_{L^q_\theta} \le C_T  \|A_0\|_{L_\theta^q}, \quad \forall t \in [0, T].
\end{align*}
When $\theta=0$,
\begin{align*} 
 \|A(t)\|_{L^q_0} \le e^{\gamma t}  \left( 1 + \| \mathcal V \|_{L^q_0} \left( e^{ \| \mathcal V \|_{L^q_0} t} - 1 \right) \right) \| A_0 \|_{L^q_0},\,\forall t\ge 0.
 \end{align*}

(ii) For all  non negative $A_0\in  \mathscr Y _{ A, B }$ with $\label{EAB}-\frac {5} {2}<B<-\frac {3} {4}<\frac {1} {2}<A<1$ there exists a unique 
$ A \in C([0, \infty); \mathscr Y _{ A, B } \cap C^1((0, \infty); \mathscr Y _{ A, B }) $ such that $A(t)\ge 0$ for all $t>0$ and satisfying equation (\ref{E1}) in $L^q_\theta( 0, \infty) $. For all $T>0$ there exists a constant $C_T>0$ satisfying,
\begin{align*}
\|A(t)\|_{A, B} \le C_T  \|A_0\|_{A, B}, \quad \forall t \in [0, T].
\end{align*}
\item \emph{(Conservation of wave action)} The solutions conserve the wave action,
\begin{align}
A_0\in L^1(0, \infty)\Longrightarrow \int _0^\infty A(t, p) \dd p= \int _0^\infty A_0(p) \dd p,\,\forall t>0.
\end{align}
\item  \emph{(Non-Conservation of energy)} There is growth of the energy along the solutions: there exists $\kappa >0$ such that
\begin{align}
A_0\in L^1((0, \infty); p^2 \dd p) \Longrightarrow \int _0^\infty A(t, p)p^2 \dd p= e^{ \kappa t}\int _0^\infty A_0(p)p^2 \dd p. 
\end{align}
\item  \emph{(Long time behaviour)} If the initial data $A_0\in \mathscr Y _{ A, B }$, $A_0\ge 0$ and $A, B$ satisfy (\ref{EAB}) then,
for all $\varphi \in C([0, \infty))\cap L^\infty((0, \infty))$,
\begin{align*}
&\lim _{ t\to \infty }\int _0^\infty A(t, p)\varphi (p) \dd p= \varphi (0) \int _0^\infty A_0(p) \dd p.
\end{align*}
\end{itemize}
\end{theo}
The last property above shows that all the mass of the perturbation $n^0(p)A(t, p)$ concentrates at the origin as $t\to \infty$.

\subsection{Organisation of the article} The content of the article is as follows. 
Section \ref{sec: NL instability} contains the proof of finite-time condensation as stated in Theorem \ref{Thm}. Next, in Section \ref{sec: linear prob proof} we proceed with the set-up of proving Theorem \ref{Theo2:Linear eq}.
In particular, Subsection  \ref{SKHW} contains the description of some properties of functions containing key information on the linearised problem. In Subsection \ref{Existence} the Cauchy problem for equation (\ref{E8}) is proved to be well-posed in several  functional spaces and conservation of the wave action (cf.(\ref{S1ETM})) is  proved to hold for some of the solutions. Results on the long time behaviour of the solutions for some initial data are presented in Subsection \ref{Slongtime}.

\section{Nonlinear instability of the RJ equilibrium} \label{sec: NL instability}

We prove in this Section Theorem 1 and to this end follow closely  \cite{XL}. Let us then denote, for $g$  a non negative measure on $[0, \infty)$,
\begin{align*} 
N_0(g, \varepsilon )  = \int_{[0,\varepsilon ] } \left( 1-\frac{x}{\varepsilon } \right)^2  \dd g(x),\,\,
N _{ \alpha  }(g, \varepsilon )=\frac {1} {\varepsilon ^\alpha }N_0(g, \varepsilon ), 
\end{align*} 
and, 
\begin{align}
&\underline{N}_\alpha (g, \varepsilon )=\inf _{ 0<\delta \le \varepsilon  } N_\alpha (g, \delta),\,\,\overline{N}_\alpha (g, \varepsilon )=\sup _{ 0<\delta \le \varepsilon  } N_\alpha (g, \delta), \label{eq: bar N}\\
&\underline{N}_{\alpha}(g) = \lim_{\varepsilon  \to 0^+}\underline{N}_{\alpha}(g, \varepsilon ),\,\,
\overline{N}_{\alpha}(g) = \lim_{\varepsilon  \to 0^+}\overline{N}_{\alpha}(g, \varepsilon ) \label{eq: overline N}.
\end{align}
Also we define
\begin{align*}
A_\alpha (g, \varepsilon  )= \varepsilon ^{-\alpha} \int_{[0, \varepsilon ] } \left( \frac{x}{\varepsilon } \right)^2  \dd  g(x).
\end{align*}
\begin{lem} \label{lem:properties of N} Suppose that $f$  is a solution of (\ref{eq:NL WKEC10}) in the weak sense  (\ref{Z2E1}) and $g$ is the measure such that $g(t,\omega )=\sqrt \omega f(t, \omega )$ Then,
\begin{itemize} 
\item[(i)] For all $ 0\leq s<t$,  $\varepsilon  >0$,
\begin{align*}
&N_\alpha(g(t), \varepsilon ) - N_\alpha(g(s), \varepsilon ) = \varepsilon ^{-\alpha} \int_s^t \mathscr K (\varphi_\varepsilon, g(\tau ) ) \dd \tau,   \\
&\mathscr K (\varphi_\varepsilon, g(\tau ) )=\iiint_{\left[  0,\infty\right)^3 }\frac{g_{1}(\tau )g_{2}(\tau )g_{3}(\tau )
\Phi }{\sqrt{\omega_{1}\omega_{2}\omega_{3}}}\times \nonumber \\
&\hskip 2cm  \times\left[  \varphi_\varepsilon \left(
\omega_{1}+\omega_{2}-\omega_{3}\right)  +\varphi_\varepsilon \left(  \omega_{3}\right)
-\varphi_\varepsilon \left(  \omega_{1}\right)  -\varphi_\varepsilon \left(  \omega_{2}\right)  \right]
\dd \omega_{1} \dd \omega_{2} \dd \omega_{3}\\
&\quad \varphi_\varepsilon  (x) := 1_{[0,\varepsilon ] }(x) \left( 1- \frac{x}{\varepsilon } \right)^2,
\end{align*} 
where $\Phi =\min \{\sqrt{\omega _1}, \sqrt{\omega _2}, \sqrt{\omega _3}, \sqrt{\omega _4} \}.$
 \item[(ii)]The functional $\bar{N}_\alpha(g,\sigma)$ is non-decreasing in the second variable:
$$ \sigma \ge \sigma '  \Longrightarrow \overline{N}_\alpha(g,\sigma) \ge  \overline{N}_\alpha(g,\sigma') $$
\item[(iii)] For any $g$ a non-negative measure on $[0, \infty)$, for all $\alpha \in [0,1/2]$ and $\varepsilon  >0$, 
$$ \underline{N}_\alpha (g, \varepsilon )  \left[A_{\frac{1-\alpha}{2}} (g, \varepsilon ) \right]^2 \leq\mathscr K (\varphi_\varepsilon, g )$$

\item[(iv)] For any $g$ a non-negative measure on $[0, \infty)$ with $g(\{0\})=0$, for all $ \alpha \geq 0$ and $\varepsilon  >0$,
$$\sqrt{N_\alpha(g, \varepsilon ) } \leq 
\int_0^1 \sigma^{\alpha/2-1} [  A_\alpha (g, \varepsilon  \sigma)  ]^{1/2} \dd \sigma. $$ 
\end{itemize} 
\end{lem}
\begin{proof} The proofs of these properties are contained in \cite{XL, XL2}. We include them here for the sake of completeness. 
The first part (i) follows directly by computing $ \frac{\dd }{\dd t} N_\alpha(g(t), \varepsilon ) $ and applying the definition of a distributional
solution of our equation for the test function $\varphi_\varepsilon$. Item (ii) is obvious from the definition of the functionals. Regarding item (iv): we may neglect the point $x=0$ as $g(\{0\})=0$ and rather work with the functional $N_0 (g,\varepsilon)  = \int_{(0,\varepsilon]} \left( 1- \frac{x}{\varepsilon } \right)^2 dg(x)$. We keep the same notation to avoid unnecessary complications. We write
\begin{align*}
 \frac{\partial }{\partial \varepsilon} N_0(g, \varepsilon ) & = 2 \int_0^\infty  \frac{\left( 1- \frac{x}{\varepsilon}\right)_+}{\varepsilon} \frac{x}{\varepsilon}  \dd g(x) \leq \frac{2}{\varepsilon} \sqrt{N_0(g, \varepsilon )} \sqrt{A_0(g, \varepsilon) }
\end{align*} 
due to Cauchy-Schwarz. So for all $\kappa >0$, 
$$ \frac{\partial }{\partial \varepsilon} \sqrt{N_0(g, \varepsilon ) + \kappa} \lesssim  \varepsilon^{-1}\sqrt{A_0(g, \varepsilon) }.$$ Now integrating over $[\kappa,\varepsilon_0]$ we get 
$$  \sqrt{N_0(g, \varepsilon_0 ) + \kappa}  -  \sqrt{N_0(g, \kappa ) + \kappa} \lesssim \int_\kappa^{\varepsilon_0}  \varepsilon^{-1}\sqrt{A_0(g, \varepsilon) } \dd \varepsilon, $$
and sending $\kappa \to 0^+$, we get 
$$ \sqrt{N_0(g, \varepsilon_0 )} 
 \lesssim \int_0^{\varepsilon_0}  \varepsilon^{-1}\sqrt{A_0(g, \varepsilon) } \dd \varepsilon 
 = \int_0^1 \sigma^{-1}\sqrt{A_0(g, \varepsilon \sigma) }    \dd \sigma.$$
 Multiply by $  \varepsilon^{-\alpha/2}$ for $\alpha>0$ and use that $A_0(g, \varepsilon \sigma) = (\varepsilon\sigma)^{\alpha} A_\alpha (g, \varepsilon\sigma)$ to prove the claim. For the item (iii), we first note that the collision operator may be rewritten as follows by decomposing the frequency space into $4$ regions: 
 \begin{align} 
& \mathscr K (\varphi_\varepsilon, g)=\iiint_{\left[  0,\infty\right)^3 }\frac{g_{1} g_{2} g_{3} 
\Phi }{\sqrt{\omega_{1}\omega_{2}\omega_{3}}} \left[  \varphi_\varepsilon \left(
\omega_{1}+\omega_{2}-\omega_{3}\right)  +\varphi_\varepsilon \left(  \omega_{3}\right)
-\varphi_\varepsilon \left(  \omega_{1}\right)  -\varphi_\varepsilon \left(  \omega_{2}\right)  \right]
\dd \omega_{1}\dd \omega_{2}\dd \omega_{3} \nonumber \\ 
& = 2 \int_{0 \leq \omega_1 <\omega_2<\omega_3} \cdots \ + 2 \int_{0 \leq \omega_2 <\omega_1<\omega_3} \cdots \ + \int_{0 \leq \omega_1 <\omega_3=\omega_2} \cdots \ + \int_{\substack{\omega_2 \geq 0 \\ \omega_3<\omega_1}} \cdots \nonumber \\
&  =  2 \int_{0 \leq \omega_1 <\omega_2<\omega_3}
 \frac{ \Delta_{\text{sym}}\varphi_\varepsilon }{\sqrt{\omega_2}\sqrt{\omega_3}} \dd g (\omega_1)\dd g (\omega_2) \dd g (\omega_3) +  \int_{0 \leq \omega_1 <\omega_3=\omega_2} \cdots \ + \int_{\substack{\omega_2 \geq 0 \\ \omega_3<\omega_1}} \cdots \nonumber \\ 
 \label{eq: K_1}
 & =  2 \int_{0 \leq \omega_1 <\omega_2\leq \omega_3} \chi_{\omega_2\omega_3}
  \frac{ \Delta_{\text{sym}}\varphi_\varepsilon }{\sqrt{\omega_2 \omega_3}} \dd g (\omega_1)\dd g (\omega_2) \dd g (\omega_3) + 
  \\ & \hspace{4cm}+ \int_{\substack{\omega_2 \geq 0 \label{eq: K_2}
   \\ \omega_3<\omega_1}} \sqrt{\frac{(\omega_2+\omega_3-\omega_1)_+}{\omega_1\omega_2\omega_3}} \Delta \varphi_\varepsilon \dd g (\omega_1)\dd g (\omega_2) \dd g (\omega_3)
 \end{align} 
 where in the third line we used the symmetries between $\omega_1$ and $\omega_2$, $\chi_{\omega_2\omega_3}$ is defined by $\chi_{\omega_2\omega_3}=21_{\omega_2<\omega_3} +1_{\omega_2=\omega_3}$ and $\Delta_{\text{sym}}\varphi_\varepsilon =\varphi_\varepsilon(\omega_3+\omega_2-\omega_1) +\varphi_\varepsilon(\omega_3+\omega_1-\omega_1) - 2\varphi_\varepsilon(\omega_3)  $ while  $\Delta \varphi_\varepsilon =\varphi_\varepsilon(\omega_1) +\varphi_\varepsilon( (\omega_2+\omega_3-\omega_1)_+) - \varphi_\varepsilon(\omega_3)-\varphi_\varepsilon(\omega_2)$. We also notice that 
 $$\Delta_{\text{sym}}\varphi_\varepsilon = (\omega_2 - \omega_1)^2 \iint_{[0,1]^2} \varphi_\varepsilon^{''} (\omega_3 + (s-t)(\omega_2 -\omega_1)) \dd s \dd t,\  \forall\ 0< \omega_1 \leq \omega_2\leq \omega_3, $$
and since $\varphi_\varepsilon^{''} (x) = 2\varepsilon^{-2}1_{[0,1]}(\varepsilon^{-1}x)$ for all $ 0\leq x <\varepsilon$, we have the lower bound 
$$\Delta_{\text{sym}}\varphi_\varepsilon \geq (\omega_2 - \omega_1)^2 \varepsilon^{-2},\  \forall\ 0< \omega_1 \leq \omega_2\leq \omega_3 \leq \varepsilon.$$
We use these to lower bound the operator in \eqref{eq: K_1} and  \eqref{eq: K_2}. Of course for both of them we have a positive sign for free, due to the convexity of the test function $\varphi_{\varepsilon}$. We use the positive sign to bound \eqref{eq: K_2}. While for \eqref{eq: K_1} we write 
\begin{align*} 
\eqref{eq: K_1}&\geq \varepsilon^{-2} \int_{0\leq \omega_1 < \omega_2 \leq \omega_3 <\varepsilon} 
 \chi_{\omega_2\omega_3} \frac{ (\omega_2 - \omega_1)^2 }{ \sqrt{\omega_2\omega_3}} 
\dd g (\omega_1)
\dd g (\omega_2) \dd g (\omega_3) \\
&  = 
\varepsilon^{-2} \int_{0< \omega_2 \leq \omega_3 <\varepsilon} 
 \chi_{\omega_2\omega_3} \frac{\omega_2^{\frac{3}{2}+\alpha}}{ \sqrt{\omega_3}} \left[\omega_2^{-\alpha} \int_{0}^{\omega_2} \left( 1- \frac{\omega_1}{\omega_2}\right)^2 \dd g (\omega_1)  \right]
\dd g (\omega_2) \dd g (\omega_3) \\
& \geq \varepsilon^{-2}  \underline{N}_\alpha (g, \varepsilon )  \int_{0< \omega_2 \leq \omega_3 <\varepsilon} 
 \chi_{\omega_2\omega_3} \frac{\omega_2^{\frac{3}{2}+\alpha}}{ \sqrt{\omega_3}} \dd g (\omega_2) \dd g (\omega_3)  \\
 &\geq   \underline{N}_\alpha (g, \varepsilon )  \left[ \varepsilon^{\frac{\alpha-1}{2}}  \int_0^\varepsilon \left(\frac{\omega_2}{\varepsilon}  \right)^{\frac{3}{2} + \alpha}    \dd g (\omega_2) \right]^2\geq \underline{N}_\alpha (g, \varepsilon )  \left[ A_{\frac{1-\alpha}{2}}(g,\varepsilon)   \right]^2, 
\end{align*} 
where in the last line we employed the more general inequality (due to the definition of $ \chi_{\omega_2\omega_3}$): 
$$\left[ \int_0^\varepsilon \omega_2^\gamma  \dd g (\omega_2) \right]^2 \leq \varepsilon^{\gamma + \beta} \iint_{0 < \omega_2 \leq \omega_3 < \varepsilon} \chi_{\omega_2\omega_3} \omega_2^{\gamma} \omega_3^{-\beta}
\dd g (\omega_2) \dd g (\omega_3)$$
for positive $\varepsilon, \gamma, \beta$, here applied to $\beta=1/2$ and $\gamma = \alpha+ 3/2$. For the last bound, we also used that $\alpha\leq 1/2$.   In total we have
$$  \mathscr K (\varphi_\varepsilon, g) \geq \underline{N}_\alpha (g, \varepsilon )  \left[ A_{\frac{1-\alpha}{2}}(g,\varepsilon)   \right]^2, $$ as required. 
\end{proof}

\begin{prop}
Suppose that $f$ is a non negative solution of (\ref{eq:NL WKEC10}) in the weak sense  (\ref{Z2E1}) that conserves the mass and energy and $g$ is the measure such that $g(t,\omega )=\sqrt \omega f(t, \omega )$, with initial data $f_0$. Then, there exists a numerical constant $C>0$ such that
\begin{align*}
g(t, \{0\})>0, \,\,\,\forall t>t_*:=C\left( \underline{N}_{1/2}(g_0) \overline{N}_{1/2}(g_0)\right)^{-1}.
\end{align*}
\end{prop}
\begin{proof} 
We are going to show that there exists a numerical constant $C>0$ such that if $g(t, \{0\})=0$ for all $t \in [0,T]$,  then
\begin{align} \label{eq:mainbound}
\underline{N}_{1/2}(g_0) \overline{N}_{1/2}(g_0) \leq C T^{-1}.
\end{align} 
We first observe that by definition for any $\lambda \in (1/2,1)$, there exists $\varepsilon _\lambda>0$ so that $\underline{N}_{1/2} (g_0, \varepsilon )>\lambda \underline{N}_{1/2}(g_0)$ for all $\varepsilon  \in (0,\varepsilon _\lambda].$ Thus combined with the item (ii) in the Lemma \ref{lem:properties of N}, $$\underline{N}_{1/2}(g(t), \varepsilon ) \geq \underline{N}_{1/2}(g_0, \varepsilon ) > \lambda \underline{N}_{1/2}(g_0).$$ Moreover from item (i) and (iii) of the same Lemma \ref{lem:properties of N}, 
\begin{align*} 
N_{\alpha}(g(t), \varepsilon ) &\geq \varepsilon ^{-\alpha}
\int_s^t \mathscr K(\varphi _\varepsilon , g(\tau )) \dd \tau
\geq \int_s^t \underline{N}_{1/2} (g(\tau ), \varepsilon )\left[ A_{\frac{1}{4} + \frac{\alpha}{2}} (g(\tau ), \varepsilon ) \right]^2 \dd \tau.   \\
& \geq  \underline{N}_{1/2} (g_0, \varepsilon )  \int_s^t \left [ A_{\frac{1}{4}+ \frac{\alpha}{2}} ( g(\tau), \varepsilon ) \right]^2
d\tau  > \lambda \underline{N}_{1/2}(g_0)   \int_s^t  \left[ A_{\frac{1}{4}+ \frac{\alpha}{2}} ( g(\tau ), \varepsilon ) \right]^2 \dd \tau.
 \end{align*} 
In other words, for all $\varepsilon  \in (0,\varepsilon _\lambda]$:
\begin{align} \label{eq:bound on time averaged A}
\int_s^t  \left[ A_{\frac{1}{2}} ( g(\tau ), \varepsilon ) \right]^2
\dd \tau < \frac{N_{\alpha}(g(t), \varepsilon ) }{ \lambda \underline{N}_{1/2}(g_0) }.
\end{align}
Now we employ item (iv) of the Lemma \ref{lem:properties of N} to get that for all $\tau \in [s,t]$
\begin{align*}
\sqrt{N_{ \frac{1}{4} + \frac{\alpha}{2}} (g(\tau ), \varepsilon  )  }  \leq \int_0^1 \sigma^{\frac{1}{8} + \frac{\alpha}{4} -1}\left[ A_{\frac{1}{4} + \frac{\alpha}{2}} ( g(\tau ), \varepsilon \sigma ) \right]^{1/2}  \dd \sigma. 
\end{align*} 
This implies 
\begin{align*} 
&\int_s^t  N_{ \frac{1}{4} + \frac{\alpha}{2}} (g(\tau ), \varepsilon  ) \dd \tau \leq
\int_s^t  \left(  \int_0^1 \sigma^{\frac{1}{8} + \frac{\alpha}{4} -1} 
\left[ A_{\frac{1}{4} + \frac{\alpha}{2}} ( g(\tau ), \varepsilon \sigma )\right ]^{1/2}
\dd \sigma  \right)^2  \dd \tau\\ 
&  \lesssim  
\int_s^t  \int_0^1 \sigma^{\frac{1}{8} + \frac{\alpha}{4} -1} 
A_{\frac{1}{4} + \frac{\alpha}{2}} ( g(\tau ), \varepsilon \sigma ) \dd \sigma  \dd \tau  = \int_0^1 
\sigma^{\frac{1}{8} + \frac{\alpha}{4} -1}   \int_s^t 
A_{\frac{1}{4} + \frac{\alpha}{2}} ( g(\tau ), \varepsilon \sigma )  \dd \tau \dd \sigma \\
& \lesssim \left( \frac{1}{2}+\alpha \right)^{-1} \int_0^1  
\sigma^{\frac{1}{8} + \frac{\alpha}{4} -1} \sqrt{t-s} \left( \int_s^t A_{\frac{1}{4} + \frac{\alpha}{2}}   ( g(\tau ), \varepsilon \sigma )^2  \dd \tau \right)^{1/2} \dd \sigma \\
& \lesssim 
\left( \frac{1}{2}+\alpha \right)^{-1}  \int_0^1 \sigma^{\frac{1}{8} + \frac{\alpha}{4} -1} \sqrt{t-s} 
\sqrt{ \frac{N_{\alpha}(g(t), \varepsilon  ) }{ \lambda \underline{N}_{1/2}(g_0)} } \dd \sigma\lesssim 
\left( \frac{1}{2}+\alpha \right)^{-2} \sqrt{t-s} \sqrt{ \frac{ \overline{N}_{\alpha}(g(t), \varepsilon  ) }{ \lambda \underline{N}_{1/2}(g_0)}},
\end{align*} 
where we applied twice Cauchy-Schwarz, Fubini, the estimate \eqref{eq:bound on time averaged A} and also utilised item (ii). Due to the monotonicity properties, this yields for all $0\leq s <t \leq T$: 
\begin{align*} 
N_{ \frac{1}{4} + \frac{\alpha}{2}} (g(s), \varepsilon  ) \leq \frac{C}{\sqrt{t-s}} \sqrt{\frac{ \overline{N}_{\alpha}(g(t), \varepsilon  ) }{ \lambda \underline{N}_{1/2}(g_0)}} \left( \frac{1}{2}+\alpha \right)^{-2},  \text{ for all } \varepsilon  \in (0, \varepsilon _\lambda], 
\end{align*} 
which due to the monotonicity of $\bar{N}_\alpha$ in the second variable, gives for all $0\leq s <t \leq T$:
\begin{align} \label{eq:bound on N}
\bar{N}_{ \frac{1}{4} + \frac{\alpha}{2}} (g(s), \varepsilon  ) 
\leq \frac{C}{\sqrt{t-s}} \sqrt{\frac{ \overline{N}_{\alpha}(g(t), \varepsilon  ) }{ \lambda \underline{N}_{1/2}(g_0)}} \left( \frac{1}{2}+\alpha \right)^{-2},  \text{ for all } \varepsilon  \in (0, \varepsilon _\lambda]
\end{align} 
for some numerical constant $C$ independent of $\alpha, \varepsilon $ and the time variables. 
Now since this holds for any $\alpha \in [0,1/2]$ we proceed by applying the estimate \eqref{eq:bound on N} iteratively where each step of iteration consists of $\alpha_i$ converging to $1/2$ from below and time step converging to $T$. To be concrete, we fix $\varepsilon  \in (0, \varepsilon _\lambda ]$, and consider $\alpha_n := \frac{1}{2} -  \frac{1}{2^n} $ and $t_k = \left( 1- 2^{-k} \right) T$, $k \in \mathbb{N}, n \in \mathbb{N}^*$.  The iteration process is to apply the estimate \eqref{eq:bound on N} for $\alpha = \alpha_{n-k}, s=t_{k-1}< t_k = t$ for $k \in [n-1]$, $n\geq 2$. We then have for each $k \in [n-1]$,  
\begin{align*} 
a_{k-1}:= \overline{N}_{ \alpha_{n-(k-1)}} (g(t_{k-1}), \varepsilon  ) 
& \leq \frac{C_1 2^{k/2}}{\sqrt{T}}  \sqrt{\frac{ \overline{N}_{\alpha_{n-k}}(g(t_k), \varepsilon  ) }{ \lambda \underline{N}_{1/2}(g_0)}} \left( 1-2^{k-n} \right)^{-2}\\
&=: b_k \sqrt{a_k}
\end{align*} 
for some numerical constant $C_1$ independent of $\alpha, \varepsilon , \lambda$ and the time variables. We define here $b_k:= \frac{C_1 2^{k/2}}{\sqrt{T}}  \frac{\left( 1-2^{k-n} \right)^{-2}}{ \sqrt{\lambda \underline{N}_{1/2}(g_0)}}$. First thing to observe is that since the mass is conserved, we have $a_{n-1} \leq m_0$, $m_0 = \int_{[0,\infty)} \dd g_0$, and thus $a_i < \infty$ for all $i \in [n-1]$. Moreover we have
\begin{align} \label{eq: iter form bounds1}
a_0 = \bar{N}_{\alpha_n}(g_0, \varepsilon ) \leq \left[ \prod_{k=1}^{n-1} b_k^{1/(2^{k-1})} \right] a_{n-1}^{1/(2^{n-1})} \leq \left[ \prod_{k=1}^{n-1}b_k^{1/(2^{k-1})} \right] m_0^{1/(2^{n-1})},
\end{align} 
while the left-hand side is lower bounded by 
$$ N_{\alpha_n}(g_0, \varepsilon )  = N_{\frac{1}{2} - \frac{1}{2^n}}(g_0, \varepsilon ) 
= \varepsilon ^{2^{-n}} N_{1/2} (g_0, \varepsilon ). $$ 
Inserting this into \eqref{eq: iter form bounds1} above, we get 
\begin{align} \label{eq: iter form bounds2}  
N_{1/2} (g_0, \varepsilon ) \leq \varepsilon ^{-2^{-n}} m_0^{1/(2^{n-1})} 
\left[ \prod_{k=1}^{n-1} 
\left( C_1 2^{k/2} \frac{\left( 1-2^{k-n} \right)^{-2}}{\sqrt{T \lambda \underline{N}_{1/2}(g_0) }} \right)^{1/(2^{k-1})} \right]. 
\end{align}
Finally we pass to the limit as $n\to \infty$, the product in the right-hand side of \eqref{eq: iter form bounds2} is converging to $$\prod_{k\geq 1} C_2^{2/2^k} 2^{k/2^k} \lim_n \prod_{k=1}^{n-1}  \left( 1-2^{k-n} \right)^{-2^{2-k}} \sim (2C_2)^2, $$
where $C_2 = C_1 (T \lambda \underline{N}_{1/2}(g_0))^{-1/2}$. Therefore, 
$$ N_{1/2} (g_0, \varepsilon ) \leq \frac{C_3}{ T \lambda \underline{N}_{1/2}(g_0)}, $$ which implies after taking the $\limsup$ in $\varepsilon $ in the left-hand side that 
$$ \bar{N}_{1/2}(g_0) \leq \frac{C_3}{ T \lambda \underline{N}_{1/2}(g_0)} $$ or since all these quantities are finite, that 
$$ \bar{N}_{1/2}(g_0)  \underline{N}_{1/2}(g_0) \leq \frac{C_3}{T \lambda}, $$
which implies \eqref{eq:mainbound} since $\lambda$ is arbitrary in $(1/2,1)$. 
\end{proof}  
\begin{cor}
\label{corcond}
There exists a numerical constant $C>0$ such that, for all $A>0$, $B>0$, and  $g_0\in \mathcal{M}^\rho _{+}\left(\left[  0,\infty\right)\right) $ for some $\rho <-1$, if 
\begin{equation*}
f_0(\omega )=A\omega ^{-1}\1 _{ 0<\omega <B }+\min (A\omega ^{-1}, g_0(\omega ))\1 _{ \omega >B },
\end{equation*}
there exists a non negative solution $f$  of (\ref{eq:NL WKEC10}) in the weak sense  (\ref{Z2E1}) that conserves the mass and energy and such that,
\begin{equation*}
g(t, \{0\})>0,\,\,\forall t>C A^{-2}.
\end{equation*}
\end{cor}
\begin{rem}
It is not known yet  if the  solutions $f$ described in Corollary (\ref{corcond}) are such that $g(t, \{0\})>0$ for all $t>0$ or if  some of them  remain a regular measure  for some time. In that case it would also be interesting to know  the behavior of $f(t, \omega )$ as $\omega \to 0$ and compare with \cite[Theorem 2.33 ]{EV3} for perturbations of the KZ spectrum $\omega ^{-7/6}$. 
\end{rem}

\section{The linearized equation in Fourier variables} \label{sec: linear prob proof}

It is useful to formulate the problem in the new variables:
\begin{align}
&\omega =p^2=e^x,\,\,p=e^{\frac {x} {2}};\,\,x=2\log p, \label{E4}
\end{align}
which then implies 
$ 2 \dd p=e^{\frac {x} {2}}\dd x= p\dd x$ or that $\frac {\dd p} {p}=\frac {\dd x} {2}.$
 The new unknown in the Fourier variables is denoted by  
\begin{align}
&A(t, p)=f(t, x). \label{E5}
\end{align}
Moreover, we define
\begin{align}
K(x)=\mathcal V(p)=\mathcal V\left(e^{\frac {x} {2}} \right)\Longrightarrow \mathcal V\left(\frac {p} {p'}\right)
=\mathcal V\left(e^{\frac {x} {2}-\frac {x'} {2}}\right)=\mathcal V\left(e^{\frac {x-x'} {2}}\right)=K(x-x').\label{E7}
\end{align}
In the new variables the equation (\ref{E1}) reads,
\begin{align}
\frac {\partial f(t, x)} {\partial t}=\gamma_N f(t, x)+\frac {1} {2}\int _\RR K(x-x')f(t, x') \dd x'. \label{E8B}
\end{align}
With a new time variable $\tau =t/2$,  and if, with some abuse of notation,  $f(2\tau , x)$ is still denoted $f(t, x)$  equation (\ref{E8B}) reads,
\begin{align}
&\frac {\partial f(t, x)} {\partial t}=\gamma f(t, x)+\int _\RR K(x-x')f(t, x') \dd x', \label{E8}\\
&\text{where},\,\,\gamma =2\gamma _N.
\end{align}
In these new variables, the following quantities may have a physical interpretation:\\
1.- Wave action of the perturbation $n_0(p)A(t, p)$,
\begin{align}
\label{S1ETM}
\int _0^\infty n_0(p) A(t, p) p^2\dd p=\int _0^\infty  A(t, p) \dd p=\frac {1} {2}\int  _{ \RR }f(t, x)e^{\frac {x} {2}} \dd x
\end{align}
2.- Total energy of the perturbation 
\begin{align}
\label{S1ETE}
\int _0^\infty n_0(p)A(t, p)p^4\dd p=\int _0^\infty A(t, p)p^2 \dd p=\frac {1} {2}\int  _{ \RR }f(t, x)e^{\frac {3x} {2}} \dd x.
\end{align}
Equation $\ref{E1}$ was studied in \cite{MM} using the Mellin transform of $\mathcal V$, which is defined as follows
\begin{align}
W_{\mathcal V}(s):=\mathcal M(\mathcal V)(s)=\int _0^\infty \mathcal V(p)p^s\frac {\dd p} {p}.\label{E11-0}
\end{align}
It was also proved that the function $W(s)=\gamma +W_{\mathcal V}(s)$ was analytic in the strip $\Re e(s)\in (-2, 5)$ with double. poles as $s=-2$ and $s=5$. In the new variables,
\begin{align}
&W_{\mathcal V}(s)=\frac {1} {2}\int  _{ \RR } K(x) e^{\frac {sx} {2}} \dd x =\frac {1} {2}\widehat K\left(\frac {is} {2}\right), \label{E11}
\end{align}
where,  $\widehat K(\xi )$ is the   Fourier transform of $K$, defined as
\begin{align}
\widehat K(\xi ) = \int_{ \RR }e^{-ix\xi }K(x) \dd x. \label{E12}
\end{align}
We have then the following relations between Mellin and Fourier transforms
\begin{align}
&\widehat K(\xi )=2W_{\mathcal V}\left( -2i\xi \right),\,\,\,W_{\mathcal V}(s)=\frac {1} {2}\widehat K\left( \frac {i s} {2}\right). \label{E12BC}
\end{align}
In a similar way as in  \cite{MM}, if Fourier transform may be applied to both sides of the equation (\ref{E8}), it follows 
\begin{align}
&\frac {\partial \widehat f(t, \xi )} {\partial t}=\gamma   \widehat f(t, \xi )+ \widehat K(\xi )\widehat f(t, \xi ).\label{Esol1}
\end{align}
This yields 
\begin{align}
\widehat f(t, \xi )&=\widehat f_0(\xi )e^{t\Omega (\xi )}, \text{ where}\\ 
\Omega (\xi )&=\gamma  +  \widehat K(\xi ) \label{Esol1B}\\
\end{align}
from where also,
\begin{align}
\Omega (\xi )=2W(-2i\xi ),\,\,\label{Esol1BB}
\end{align}
and the function $\Omega  $ is well-defined and analytic in the strip $\Im m(\xi )\in (-1, 5/2)$. 
We denote, for all $t>0$, the following functions
\begin{align}
&G(t) :=\mathscr F^{-1}\left( e^{t\Omega }\right)\label{S1Hdef}\\
&H(t) :=\mathscr F^{-1}\left( e^{\widehat  Kt}-1\right)\label{S1Wdef},
\end{align}
where $\mathscr F^{-1}$ denotes the inverse Fourier transform defined as,
\begin{align*}
\mathscr F^{-1}(g)(x)=\frac {1} {2\pi }\int  _{ \RR }e^{ix\xi }g(\xi )d\xi .
\end{align*} 

As observed in \cite{MM}, equation (\ref{E8}) has two trivial stationary solutions,
\begin{prop} \label{prop: station solutions_zeros of W}
The functions $e^{-x}$ and $e^{-x/2}$  are stationary solutions of equation (\ref{E8}). 
\end{prop}
\begin{proof}
If $A_k(p)=p^{-k}$ then  $\mathcal M(A)(s)=2\pi \delta (i(s-k))$. Since $W(1)=W(2)=0$
 it follows that $\mathcal M(A_k)(s)W(s) = 0$ for all $s$. 
Then the Mellin image of the solution $\mathcal M(A_k)(s)$ satisfies  $\mathcal M(A_k)_t=\mathcal M(A_k)(s)W(s) = 0$. 
Taking inverse Mellin transform, $A_k$ is a stationary solution of equation (\ref{E1}) for $k=1, 2$. 
It follows that $e^{-\frac {x} {2}}$  and  $e^{-x}$ are solutions of \eqref{E8}.
 \end{proof}

\subsection{Properties of $K$, $\widehat K$, $H$ and $G$.}
\label{SKHW}
As with the Mellin transform  for equation (\ref{E1}), to solve the Cauchy problem for \eqref{E8} with general initial data, use of the inverse Fourier transform applied to  equation \eqref{Esol1B} appears to be the most direct approach. This method relies on certain properties of $K$, $\widehat K$ and $H$. Some of them directly  follow from those of $\mathcal V$ and $W _{ \mathcal V }$ obtained in \cite{MM}, others  are demonstrated in what follows in this subsection. 
 
Even though the necessary details for this article are presented in the first part of the Appendix \ref{Sapp}, let us summarise here briefly the basic properties. 
The function $\mathcal V\in C^\infty ((0, 1)\cup (1, \infty))$ has a singularity at $p=1$ of order $|p-1|^{-1/2}$. It follows from the known detailed behaviours of $\mathcal V(p)$ as $p\to 0$ and $p\to \infty$ that  $ \mathcal V\in L^p(0, \infty)$ for $p\in [1, 2)$. 

The function $W=\gamma +W_{\mathcal V}$, where $W_{\mathcal V}$  is the Mellin transform of $\mathcal V$, is meromorphic with two double poles at $s=-2$, $s=5$, and analytic on the strip $\Re e(s)\in (-2, 5)$. Its behaviours as $s\to -2$ and $s\to 5$ are also known in detail  (cf. (\ref{ED2W}), (\ref{ED3W})).

The corresponding expressions and properties  for the functions $K$ follow from the change of variables $p=e^{x/2}$. The behaviour of $K, K'$ and $K''$ as $|x|\to \infty$ and $x\to 0$ are given also in the Appendix \ref{Sapp}. In particular $K\in L^p(\RR)\cap L^p\left(e^{x/2}\right)$ for all $p\in [1, 2)$ with $ \| K \| _{   L^p\left(e^{x/2}\right)}=2 \| \mathcal V \|_p$.

The following property of $H$ easily follows from the integrability properties of $K$.

\begin{prop}
\label{S1PH1}
For all $t>0$,
\begin{align*}
&H(t) = \mathscr{F}^{-1}\left( e^{\widehat{K} t} - 1 \right) \in L^p(\mathbb{R}), \quad \forall \, p \in [1, 2)\\
 \text{and } \ & \| H(t) \|_p \le  \| K \|_p\left( e^{ \| K \|_1t}-1\right).
\end{align*}
\end{prop}
\begin{proof}

Since $K\in L^1(\RR)$, by the Riemann-Lebesgue Lemma, $ e^{\widehat K(\xi )t}-1 \to 0$ as $|\xi |\to \infty$. It follows that for each $t>0$, $e^{\widehat K(\cdot )t}-1\in L^\infty (\RR)\subset \mathscr S'(\RR)$ and tends to zero as $|\xi |\to \infty$. Its inverse Fourier transform is then well defined and belongs to $\mathscr S'(\RR)$. Moreover,
\begin{align*}
&e^{\widehat  K(\xi )t}-1=\sum _{ n=1 }^\infty \frac {\widehat K(\xi )^nt^n} {n!}
\end{align*}
while 
\begin{align}
H(t, x)&=\mathscr F^{-1}\left( e^{\widehat Kt}-1\right)=\sum _{ n=1 }^\infty \mathscr F^{-1}\left(\widehat K^n\right)\frac { t^n} {n!}\nonumber\\
&=\sum _{ n=1 }^\infty  (\underbrace{K\ast K\ast K\ast \cdots \ast K} _{ n\,\text{times} }) (x) \frac { t^n} {n!}. \label{eq: K n convolut}
\end{align}
Since $K\in L^p$ for $p\in [1, 2)$,
then $K\ast K\in L^p(\RR)$ too and
\begin{align*}
 \| K\ast K \|_p \le  \| K \|_1 \| K \|_p
\end{align*}
 for all $p\in [1, 2)$. Then inductively
\begin{align*}
&\underbrace{K\ast K\ast K\ast \cdots \ast K} _{ n\,\text{times} }\in  L^p(\RR),\, \quad \forall \, p \in [1, 2)
\end{align*}
and
\begin{align*}
&\| \underbrace{K\ast K\ast K\ast \cdots \ast K} _{ n\,\text{times} } \|_p\le \|K\|_1^n \|K\|_p.
\end{align*}
Consequently from \eqref{eq: K n convolut}, 
\begin{align*}
&\left\| H(t)\right\|_p\le \|K\|_p\sum _{ n=1 }^\infty \frac { \| K\|_1^n  t^n} {n!}=
\|K\|_p\left (e^{ \| K \|_1 t} -1\right). 
\end{align*}
This concludes the stated estimate on $H(t)$.
\end{proof}

\begin{prop}
\label{ProphatK}
For all compact interval $I \subset (-1, 5/2)$ and $\varepsilon > 0$, there exist constants $C _{ I, 1 }$ and $C _{ I, 2 }$ 
depending on $I$ and $\varepsilon$ such that, for all $\xi \in \CC$ with $\Im m\xi \in I$,
\begin{align}
& \left\vert \widehat K(\xi )-\pi  \sqrt { 2 \pi }\xi ^{-1/2} +
\left( K(\varepsilon )i e^{-i\varepsilon \xi }+
K(-\varepsilon)  i e^{i\varepsilon \xi} \right)
\xi^{-1}\right\vert \le C_{ I, 1 } |\xi |^{-3/2}, \label{pgcd.e0} 
\end{align}
and 
\begin{align}
&  \left\vert  \widehat K(\xi ) \right\vert 
\le C _{ I, 2 }(1+\xi)^{-1/2}, \,\,\forall \, \xi. \label{pgcd.e1}
\end{align}
\end{prop}

\begin{proof} Since $K\in L^1(\RR)$,
\begin{align}
&\widehat K(\xi )=\int  _{ \RR }K(x)e^{-ix\xi }\dd x
=\int  _{ |x|>2  }K(x)e^{-ix\xi } \dd x
+ \int  _{ |x|<2  }K(x)e^{-ix\xi } \dd x,\nonumber
\end{align}
and for the first term we write
 \begin{align}
\int_{  2} ^\infty  K(x)e^{-ix\xi }\dd x&=
\frac {i} {\xi }  \int  _ 2 ^\infty 
K(x)\left(e^{-ix\xi } \right)_x \dd x\nonumber\\
&= -\frac {i} {\xi }\int  _{ 2  }^\infty K'(x) e^{-ix\xi }\dd x
+ \lim _{ x\to \infty  }K(x)\frac {i e^{-ix\xi }} {\xi }-
K(2)\frac {i e^{-i2 \xi }} {\xi }. \nonumber
\end{align}
Using integration by parts again gives the existence of a constant $C>0$ such that for $\Im m \xi \in (-1, 5/2)$,
\begin{align*}
&-\frac {i} {\xi }\int  _{ 2 }^\infty K'(x) e^{-ix\xi }\dd x = 
- \frac {1} {\xi ^2} \left( \int  _{ 2 }^\infty K''(x)e^{-ix\xi }\dd x -
K'(2) e^{-i 2\xi  } \right)
\end{align*} and so 
\begin{align*}
&\left| \int  _{ 2 }^\infty K''(x)e^{-ix\xi }\dd x-
K'(2)e^{-i x\xi  }\right|\le C \int  _{ 2 }^\infty  |K''(x)| 
e^{x\Im m\xi  }\dd x + | K '(2) |  e^{2\Im m\xi }<C.
\end{align*}
The last bound follows from the behaviour of $K''(x)$ at infinity given in Appendix \ref{Sapp} and 
the fact that $\Im m \xi < 5/2$. 
By \eqref{E13}, when $x\to -\infty$,
\begin{align}
\left\vert K(x)\frac {i e^{-ix\xi }} {\xi }\right\vert 
\le C |\xi |^{-1}|x|e^{x(1+\Im m\xi )}, \label{E16}
\end{align}
and by \eqref{E14} when $x\to \infty$,
\begin{align}
\left\vert  K(x)\frac {i e^{-ix\xi }} {\xi }\right\vert \le C |\xi |^{-1}|x|e^{x(-\frac {5} {2}+\Im m\xi )}. \label{E16}
\end{align}
Then, for $\xi $ such that $\Im m \xi \in (-1, 5/2)$,
\begin{align}
\lim _{x \to \infty} K(x) \frac {i e^{-ix\xi }} {\xi }
= \lim _{x \to -\infty} K(x) \frac {i e^{-ix\xi }} {\xi }=0. \label{E17}
\end{align}
On the other hand,  by the behaviour of $K$ near $x=0$ given in \eqref{E14b} for $x\in (0, 2)$,
$K(x)= \frac {\pi } {\sqrt {|x|}}+\mathcal O(x)^{1/2}$, we have that 
\begin{align*}
\int _0^2 K(x)e^{-ix\xi }dx 
&=\pi 
\int_0^2 
\frac {e^{-ix\xi }\dd x}{\sqrt {x}} +\mathcal O\left( \int_0^2x^{1/2}e^{-ix\xi }\dd x \right),\, |\xi |\to \infty\\
&\sim  A \pi \xi^{-\frac {1} {2}} + \mathcal O(\xi ^{-3/2}),\,\,|\xi| \to \infty
\end{align*}
where 
\begin{align*}
A:=\int _0^\infty z^{-\frac {1} {2}}e^{-i z}\dd z = 
e^{-\frac {\pi  i} {4}}\Gamma \left(\frac{1}{2}\right).
\end{align*}
It follows that for each $2>0$, there exists a constant $C>0$ such that for $\Im m \xi \in (-1, 5/2)$,
\begin{align*}
\left|\int_0^\infty K(x)e^{-ix\xi } \dd x-A \pi \xi ^{-\frac {1} {2}}+K(2)\frac {i e^{-i2\xi }} {\xi }\right|
\le C |\xi|^{-3/2},\,\,|\xi| \to \infty.
\end{align*}
After using a similar argument on $(-\infty, 0)$: 
\begin{align}
&\int_{-\infty }^{-\varepsilon} K(x)e^{-ix\xi } \dd x =
\frac {i} {\xi }  \int _{ -\infty }^{- \varepsilon} 
 K(x)\left(e^{-ix\xi } \right)_x \dd x \nonumber\\
&= -\frac {i} {\xi }\int_{-\infty}^{-\varepsilon}  K'(x) e^{-ix\xi } \dd x+
K(-2)\frac {i e^{i2\xi }} {\xi} - 
\lim _{ x\to \infty  }K(x)\frac {i e^{-ix\xi }} {\xi} \nonumber\\
&=-\frac {i} {\xi }\int_{-\infty}^{-\varepsilon}  K'(x) e^{-ix\xi }\dd x+
K(-2)\frac {i e^{i2\xi }} {\xi }, \nonumber
\end{align}
where for $\Im m \xi \in (-1, 5/2)$,
\begin{align*}
&-\frac {i} {\xi }\int  _{-\infty }^{-2} K'(x) e^{-ix\xi }\dd x
= -\frac {1} {\xi ^2} \left( \int _{-\infty }^{-\varepsilon} K''(x)e^{-ix\xi }\dd x+
K'(-2)e^{i 2\xi  }\right)
\end{align*} and where 
\begin{align*}
&\left| \int_{-\infty }^{-2} K''(x)e^{-ix\xi } \dd x+
K'(-2)e^{i 2\xi  }\right|\le 
C\int_{ -\infty  }^{-2} |K''(x)| e^{x\Im m\xi  } \dd x
+|K'(-2)| e^{\varepsilon \Im m\xi }<C.
\end{align*}
One then arrives at the existence of a constant $C>0$ such that for  $\Im m \xi \in (-1, 5/2)$,
\begin{align*}
&\left|\int  _{ -\infty }^\infty K(x)e^{-ix\xi }dx- (A+A')\pi  \xi ^{-\frac {1} {2}}+K(2)\frac {i e^{-i2\xi }} {\xi }+
K(-2)\frac {i e^{i2\xi }} {\xi }\right|\le  C|\xi|^{-3/2},\,|\xi| \to \infty,
\end{align*}
where 
\begin{align*}
&A' = \int_0^\infty z^{-1/2}e^{iz} \dd z=e^{\frac {\pi  i} {4}}\Gamma \left( \frac{1}{2}\right). 
\end{align*}
It follows since $A+A'=\sqrt { 2 \pi }$, that 
\begin{align}
\left| 
\widehat K(\xi )-\pi  \sqrt { 2 \pi } \xi ^{-1/2} + K(2)\frac {i e^{-i2\xi }} {\xi }+
K(-2)\frac {i e^{i\varepsilon \xi }} {\xi }\right|\le C|\xi |^{-3/2} \label{E17b}
\end{align}
which concludes \eqref{pgcd.e0}.
Since $\widehat K\in L^\infty(\RR)$, \eqref{pgcd.e1} follows.
\end{proof}

Let us remark now that from the analyticity properties of $W$ in \cite{MM} It follows that $\widehat K$ is then analytic in the strip $\Im m \xi \in (-1, 5/2)$, it has poles $\xi =-i$ and $\xi =5i/2$ and
\begin{align*}
\widehat K(\xi )&=2W\left( -2i\xi \right)\sim  \frac {2} {(-i\xi +1)^2},\,\,\xi \to -i\\
\widehat K(\xi )&=2 W\left( -2i\xi \right)\sim  \frac {8} {(2i\xi + 5)^2},\,\,\xi \to \frac {5 i} {2}.
\end{align*}

The behaviour of $H(t)$ for large values of $|x|$ and $t$ in bounded intervals of $[0, \infty)$ is given in the following result.

\begin{cor}
\label{H2}
Let $H$ be the distribution defined in (\ref{S1Wdef}). Then for all $t>0$, 
\begin{align*}
H(t)=t K+\frac { t^2(K\ast K)} {2}+\text{continuous and bounded function on }\,\RR.
\end{align*}
Moreover,  for any $\beta\in (-3, -1)$, $\alpha\in \left(5/2, 9/2\right)$  and  $t>0$,
\begin{align}
\label{S3H7}
 H(t, x)-t K(x)+\frac{ t^2(K\ast K)(x)}{2} = 
\begin{cases}
\ell_1(t)e^x+\mathcal O\left(t^3e^{-\beta x} \right),\,\,x\to -\infty\\
\ell_2(t)e^{-\frac {5x} {2}}+\mathcal O\left(t^3e^{-\alpha  x} \right),\,\,x\to \infty
\end{cases}
\end{align}
where
\begin{align*}
&\ell_1(t)=\operatorname{Res}\left( P(t), \xi =-i \right),\,\,\,
\ell_2(t)=\operatorname{Res}\left( P(t), \xi =- \frac{5i}{2}\right)\\
\end{align*}
with
\begin{align*}
P(t)=e^{t\widehat{K} }-1-t\widehat{K} - \frac {t^2\widehat{K}^2}{2}.
\end{align*}
There exist positive constants, $C_H, D_H$, such that for all $t>0$
\begin{align}
\label{S3H8}
&|H(t, x)|\le 
\begin{cases}
 C_Hte^{D_Ht}e^x(1+ |x|^3),\,\,\forall\ x<0, ,\,\,\forall\ t>0.\\
 C_Hte^{D_Ht}e^{-\frac {5x} {2}}(1+ |x|^3),\,\,\forall\ x>0, \,\,\forall\ t>0
\end{cases}
\end{align}
\end{cor}

\vfill
\eject

\begin{proof} It follows from (\ref{pgcd.e1}) that for all $t\in I$,
\begin{align*}
P (t, \xi )=e^{t\widehat K(\xi )}-1-t\widehat K(\xi )-\frac {t^2\widehat K^2(\xi )} {2}=\mathcal O\left(t^3 |\xi |^{-3/2}\right),\,|\xi |\to \infty
\end{align*}
so that $P(t)\in L^1(\RR)\cap L^\infty(\RR)$ and it inverse Fourier transform is a bounded continuous function. On the other hand
\begin{align*}
\mathscr F^{-1}\left(t\widehat K+\frac {t^2} {2} \widehat K^2 \right)= t K+\frac { t^2} {2}K\ast K, 
\end{align*}
and by definition of $H$, 
\begin{align*}
H(t, x)=  tK(x)+\frac { t^2} {2}(K\ast K)(x)+\text{continuous and bounded function.}
\end{align*}
For the next part of the Proposition, the behaviour of $\mathscr F^{-1}(P(t))(x)$ as $x\to \pm \infty$ are given by classical arguments. Indeed,
\begin{align*}
\mathscr F^{-1}\left(P(t)\right)(x)&= \frac {1} { 2\pi }\int  _{ \RR } P(t, \xi)e^{ix\xi }\dd \xi 
\end{align*}
where the integral is absolutely continuous. The function $P(t, \cdot)$ is analytic in the strip $\Im m \xi \in (-1, 5/2)$ with two double poles at $\xi =-i$ and $\xi =5i/2$. It is also integrable since by Proposition \ref{ProphatK}, $|P(t, \xi )|\le Ct^3(1+|\xi |)^{-3/2}$. Therefore lassical contour deformation argument and location of the poles of the function $\widehat K$ give,
\begin{align*}
\int_{ \RR } P(t, \xi)e^{ix\xi }\dd \xi &=2\pi i\operatorname{Res}\left( P(t), \xi =
-i\right)e^x+\int_{\Im m \xi =\beta } P(t, \xi)e^{ix\xi }\dd\xi,\,\,x\to -\infty
\end{align*}
for all $\beta\in (-3, -1)$, and 
\begin{align*}
\int_{ \RR } P(t, \xi)e^{ix\xi }\dd \xi &= 2\pi i\operatorname{Res}\left(P(t), \xi =
\frac {5i}{2}\right)e^{-\frac {5x}{2}}+\int_{\Im m \xi =\alpha} P(t, \xi)e^{ix\xi }\dd\xi,\,\,x\to \infty,
\end{align*}
 for all $\alpha \in \left(\frac{5}{2}, 9/2\right)$. Then
\begin{align*}
\left|\int  _{\Im m \xi=\beta }P(t, \xi)e^{ix\xi }\dd\xi\right|\le Ce^{-\beta x}\int_{ \Im m \xi =\beta} |P(t, \xi )|\dd\xi\le 
Ce^{-\beta x}t^{3}\\
\left|\int  _{\Im m \xi=\alpha  } P(t, \xi)e^{ix\xi }\dd\xi\right|\le Ce^{-\alpha  x}\int_{ \Im m \xi =\alpha} |P(t, \xi )|\dd\xi
\le Ce^{-\alpha  x}t^3
\end{align*}
and \eqref{S3H7} follows. By \eqref{E12BC} and \eqref{ED2W},
\begin{align*}
& \widehat K(\xi ) \sim \left( \frac {2} {(1-i\xi )^2}+ \frac {8} {3(1-i\xi )}+\mathscr A(\xi )\right),\,\xi \to -i\\
&\mathscr A(\xi ):\text{analytic near}\,\,\xi =-i,\,\,\,\mathscr A(\xi )=\mathcal O(1),\,\,\xi \to -i.
\end{align*}
It follows that for all integers $n>1$,
\begin{align*}
\widehat K(\xi )^n &=\left( \frac {2} {(1-i\xi )^2}+ \frac {8} {3(1-i\xi )}+\mathscr A(\xi )\right)^n,\,\xi \to -i
\end{align*}
and we then compute
\begin{align*}
\operatorname{Res}\left( \widehat K(\xi )^n, \xi =-i\right)=\frac {8 in} { 3}\,\mathscr A(-i )^{n-1}.
\end{align*}
Then,
\begin{align*}
\sum_{n=3}^\infty \frac {t^n} {n!}\operatorname{Res}
\left(  \widehat K(\xi )^n, \xi =-i\right) &= \frac {8i} {3}\sum _{n=3}^\infty \frac {t^n \mathscr A(-i)^{n-1}} {(n-1)!}
\\
&= \frac {8 i t} {3 }\left(
e^{\mathscr A(-i)t}-1-\mathscr A(-i)t \right)
\end{align*}
and, for some constants $C>0$, $D>0$,
\begin{align}
\label{H2E1R}
\left|\sum _{n=3 }^\infty \frac {t^n } {n!}\text{Res}\left(  \widehat K(\xi )^n, \xi =-i\right)\right|
\le Ct^2e^{Dt},\,\,\forall\ t > 0.
\end{align}
A similar argument analogously gives
\begin{align*}
\sum _{ n=3 }^\infty \frac {t^n} {n!}\operatorname{Res}
\left( P(t), \xi = \frac {5i}{2}\right)&= \frac {8 it} {3} \sum_{n=3 }^\infty \frac {t^n \mathscr A(5i/2)^{n-1}}{ (n-1)!}
\\&=\frac {8it} {3 }\left(e^{\mathscr A(5i/2)t}-1-\mathscr A(5i/2)t \right)
\end{align*}
and, for some constants $C'>0$, $D'>0$,
\begin{align}
\label{H2E1R2}
\left|\sum_{ n=3 }^\infty \frac {t^n} {n!}
\operatorname{Res}\left( P(t), \xi =\frac {5i} {2}\right)\right|\le C't^2e^{D't},\,\,\forall\ t > 0.
\end{align}

On the other hand, $K\ast K$ is easily estimated as follows. For all $x>0$,
\begin{align*}
|(K\ast K)(x)|\le &Ce^{-\frac {5x} {2}} \int_{ -\infty }^0e^{\frac {7y} {2}}(x-y)y \dd y+Ce^{-\frac {5x} {2}}\int  _{0}^x (x-y) y \dd y\\
&+Ce^x\int_x^\infty   e^{-\frac {7y} {2}}(x-y)y \dd y
\le Ce^{-\frac {5x} {2}}(1+x)+Ce^{-\frac {5x} {2}}x^3.
\end{align*}
For $x<0$,
\begin{align*}
|(K\ast K)(x)|\le Ce^{-\frac {5x} {2}} \int_{ -\infty }^xe^{\frac {7y} {2}}(x-y)y \dd y
+Ce^{x}\int  _{x}^0  (x-y) y \dd y+\\
+Ce^x\int_0^\infty  e^{-\frac {7y} {2}}(x-y)y \dd y
\le Ce^{x}(1+|x|)+Ce^{x}|x|^3.
\end{align*}
This ends the proof of \eqref{S3H8}, and of Corollary \ref{H2}.
\end{proof}

Before we conclude this subsection, let us state and prove one more result that provides a 
detailed description of the behaviour of $W(t)$ as $|x|\to 0$ and $|x|\to \infty$. 
While this fact is not necessary for the purposes of this article 
- namely, well-posedness and the long-term behaviour of the linear problem -
 we include it here as we believe it will be valuable for future studies on the nonlinear problem. 
\begin{cor} 
\label{hatK2}
The function $G(t)$ defined in (\ref{S1Hdef}) admits the following decompositions:
\begin{itemize} 
\item[(i)] For all $t>0$ in a bounded set,
\begin{align*}
G(t)= J_1(t)+Q_1(t)
\end{align*}
where
\begin{align}
 J_1(t, x)=e^{\gamma t}\delta_0(x)+\alpha(t)\operatorname{sign} (x)+ \beta(t)\left(\operatorname{sign}(x-2)+\operatorname{sign}(x+2)\right)
\label{hatK2E1}
\end{align}
and where $Q_1(t)$ is a H\"older continuous function in a neighbourhood of $x = 0$. The functions $\alpha, \beta$ and the H\"older constants of $Q_1(t)$ are uniformly bounded in bounded intervals of $t$. 
\item[(ii)] There exist constants $\rho_0>0$, $\rho_1\in (0, 1)$ as close to $1$ as desired, $\rho _2\in (0, 5/2)$ as close to $5/2$ as desired and a function $Q_2$ H\"older continuous in a neighbourhood of $x = 0$ such that, for some constant $C>0$,
\begin{align}
 \label{hatK2E2}
& G(t)=J_2(t)+Q_2(t)
\end{align}
with
\begin{align}
&\left| J_2(t, x)-e^{-\gamma t}\delta _0(x) \right| \le C e^{-\rho _0 t}\varphi (x),\\
&|\varphi (x)|\le 
\begin{cases}
e^{\rho_1x},\,&x<-1\\
e^{-\rho_2x},\, &x>1\\
|x|^{-\frac {1}{2}},\,&|x|<1.
\end{cases}  \label{hatK2E3}
\end{align}
\end{itemize} 
\end{cor}

\begin{proof}
For all $t$, by Taylor expanding the exponential function and using the Proposition \ref{ProphatK} we have, 
\begin{align*}
&e^{(\gamma +\widehat K(\xi ))t}=e^{\gamma t}\left(e^{ \widehat K(\xi )t} \right)=e^{\gamma t}
\sum _{ n=0 }^\infty \frac {\widehat K(\xi )^n t^n} {n!}
\end{align*}
with
\begin{align*}
&\sum_{n=0 }^\infty \frac { \widehat K(\xi )^n  t^n} {n!}
=1+ \widehat K(\xi )  t+\frac {\widehat K(\xi)^2 t^2}{2}+\mathcal O\left(\widehat K(\xi )   t \right)^3
=g_1(t, \xi )+h_1(t, \xi )
\end{align*}
and where
\begin{align*}
&g_1(t, \xi )=1+\pi  \sqrt { 2 \pi }\,  t |\xi| ^{-1/2}+2\pi ^3 \frac { t^2} {\xi }-t^2 \left(K(\varepsilon )i e^{-i\varepsilon \xi }+
K(-\varepsilon )  i e^{i\varepsilon \xi }\right) \xi^{-1}\\
&h_1(t, \xi )=\mathcal O\left(\widehat K(\xi )  t \right)^3=\mathcal O\left(|\xi |^{-1/2} t \right)^3,\,\,|\xi |\to \infty.
\end{align*}
Then,
\begin{align*}
&G(t, x)=\mathscr F^{-1}\left( e^{(\gamma +\widehat   K)t}\right)(x)=e^{\gamma t}\delta _0(x)+\mathscr F^{-1}\left(g(t)+h(t)  \right)(x)
\end{align*}
from where \eqref{hatK2E1} follows. On the other hand, consider a regular cut-off function $\chi$ such that $\chi (s)=1$ for $|s|\le 1$ and $\chi (s)=0$ for $|s|>2$. Then,
\begin{align}
\mathscr F^{-1}\left( e^{(\gamma +\widehat  K)t}\right)=e^{\gamma t}\Bigg[\mathscr F^{-1}(1)+&
\mathscr F^{-1}\left( e^{\widehat Kt}  \left[1-\chi\left(\frac {|\xi |} { t^2} \right) \right]-1\right)\nonumber\\
&\hspace{3cm}+\mathscr F^{-1}\left( e^{\widehat Kt} \chi\left(\frac {|\xi |} { t^2}\right) \right) \Bigg]. \label{hatK2E5}
\end{align}
When $|\xi |< t^2$ the first and second terms in the right-hand side of \eqref{hatK2E5} cancel. When $|\xi |>2 t^2$, in the second term in the right-hand side of (\ref{hatK2E5}) is
\begin{align*}
&e^{\widehat K(\xi )t}  \left[1-\chi\left(\frac {|\xi |} {t^2} \right) \right]-1=\left(e^{\widehat K(\xi ) t}-1\right)\\
&\text{with} \,\,|\widehat K(\xi )|t\le  C\frac { t} {|\xi |^{1/2}}.
\end{align*}
Then, for $t\gg1$ and $|\xi |>2t^2$, 
\begin{align*}
\left(e^{\widehat K(\xi ) t}-1\right)=\sum _{ n=1 }^\infty\frac {\widehat K(\xi )^n t^n} {n!}& =
\widehat K(\xi )  t+\widehat K(\xi )^2  t^2+\mathcal O\left(|\xi |^{-1/2}\,   t\right)^3\\
&=g_2(t, \xi )+h_2(t, \xi )
\end{align*}
where
\begin{align*}
&g_2(t, \xi )=\pi  \sqrt { 2 \pi } \frac { t } {|\xi| ^{1/2}}+2\pi ^3 \frac { t^2} {\xi }-\frac {i   t^2} {\xi }\left(K(2) e^{-i2\xi }+
K(-2) e^{i2\xi }\right)\\
&h_2(t, \xi )=\mathcal O\left(|\xi |^{-1/2}t\right)^3.
\end{align*}
Then we compute
\begin{align*}
&\mathscr F^{-1}\left(e^{\widehat  K(\xi ) t}-1\right)=\pi  \sqrt { 2 \pi } \frac { t} {|x|^{1/2}}+\sqrt 2\, \pi ^{7/2} t^2
\operatorname{sign} (x) \\
&\hspace{4cm} -\sqrt{\frac {\pi } {2}} t^2
\left(K(2)\operatorname{sign} (x-2)- K(-2)\operatorname{sign} (x+2)\right)+ H_2(t, x).
\end{align*}

Since $\widehat K$ is analytic in the strip $\Im m\xi \in (-1, 5/2)$, the third term in the right-hand side of (\ref{hatK2E5}) is estimated  using contour deformation. When $x\to -\infty$, for all $\rho >-1$ as close to $-1$ as desired,
\begin{align*}
&\int_{ \RR }e^{ix\xi } e^{\widehat K(\xi ) t}\chi\left(\frac {|\xi |} {t^2}\right)d\xi=\int  _{ \Im m \xi =\rho  }e^{ix\xi } e^{\widehat K(\xi )t}\chi\left(\frac {|\xi |} {t^2}\right) \dd \xi
\end{align*}
and where 
\begin{align*}
&\left|\int_{ \Im m \xi =\rho} e^{ix \xi } e^{\widehat K(\xi )t} \chi\left(\frac {|\xi |} {t^2}\right) \dd \xi\right|
\le Ce^{-\rho x} \int_{ \substack{ \Im m \xi =\rho, \\ |\xi |<2t^2}} e^{|\widehat K(\xi )|t}\dd \xi \\
& \le  Ce^{-\rho  x } \int  _{ \substack{\Im m \xi =\rho,\\ |\xi |<2 t^2 } } e^{ t(1+|\xi |)^{-1/2}}\dd \xi=  Ce^{-\rho  x } \int_{|u |<2 t^2  } e^{ t(1+|u+i\rho  |)^{-1/2}}\dd u.
\end{align*}
Now if $ t^2(1+u)^{-1}=v$, $1+u= t^2v^{-1}$ and $\dd u=- t^2v^{-2}\dd v$. Thus, 
\begin{align*}
e^{\gamma t}\int  _{ \RR }e^{ix\xi } e^{(\widehat K(\xi )-\gamma )t}d\xi &\le 
2Ce^{\gamma t}e^{-\rho  x } \int  _0^{2 t^2   } e^{  t(1+u)^{-1/2}}\dd u\\
&=2Ce^{\gamma t}e^{-\rho x}\ t^2\int _{ \frac { t^2} {1+2 t^2} }^{ t^2}e^{\sqrt v}\frac {\dd v} {v^2}.
\end{align*}
These integrals are found to be 
\begin{align*}
\tau^2\int  _{ \frac {\tau^2} {1+2\tau ^2} }^{t\tau^2}e^{\sqrt v}\frac {\dd v} {v^2}=- e^t\tau (1+\tau ) +e^{\frac {t\tau } {(1+2\tau^2}^{1/2}} (1+2\tau^2+\tau (1+2\tau^2)^{1/2} +\\
+\tau ^2\left(\text{ExpIntegralEi}(\tau )-\text{ExpIntegralEi}\left(\frac {\tau } {(1+2\tau ^2)^{1/2}}\right)\right)\le Ce^{\tau }
\end{align*}
and since $\gamma >1$, there exists $\delta _1>0$ such that
\begin{align*}
e^{\gamma t}\int  _{ \RR }e^{ix\xi } e^{(\widehat  K(\xi )-\gamma )t}d\xi \le  Ce^{-\delta _1\kappa  t}e^{-\rho x}.
\end{align*}

A similar argument gives, for $x\to \infty$ the existence of $\delta _2 >0$ and $\rho '<5/2$ as close to $5/2$ as desired, such that
\begin{align*}
e^{\gamma t}\int  _{ \RR }e^{ix\xi } e^{(\widehat  K(\xi )-\gamma )t}\dd \xi &  
\le Ce^{- \delta _2 \kappa t}e^{-\rho 'x}
\end{align*}
and part (ii) of  Corollary \ref{hatK2} follows.
\end{proof}

\subsection{Existence,  uniqueness and properties of solutions to  \eqref{E8}.}
\label{Existence}

Using the information obtained in the previous subsection, it is now easy to obtain solutions to \eqref{E8}.

\begin{prop} 
\label{prop1}
For all \( f_0 \in \mathscr{S}'(\mathbb{R}) \) and \( t > 0 \), \( e^{\Omega  t}\, \widehat{f}_0  \in \mathscr{S}'(\mathbb{R}) \) and \( \widehat{K} \widehat{f}_0 e^{\Omega  t} \in \mathscr{S}'(\mathbb{R}) \). Then, the distribution
\begin{align}
f(t) = \mathscr{F}^{-1}\left(e^{t\Omega }\, \widehat{f}_0  \right) =G(t)\ast  f_0  \in \mathscr{S}'(\mathbb{R}) \label{prop1.e2}
\end{align}
satisfies
\begin{align}
\frac{\partial f(t)}{\partial t} = \gamma f(t) + K\ast f(t) \in \mathscr{S}'(\mathbb{R}) \quad \text{for all } t > 0. \label{prop1.e3}
\end{align}
\end{prop}

\begin{proof}
The idea is to write,
\begin{align*}
e^{(\gamma  +\widehat  K(\xi ))t}=e^{\gamma t} e^{\widehat K(\xi ))t}=e^{\gamma  t} \left(1+\left(e^{\widehat K(\xi )t}-1\right)\right)
\end{align*}
with  $\mathscr F^{-1}\left(e^{\gamma  t} \right)=e^{\gamma  t}\delta _0$.
Since  $\widehat K\in L^\infty$, $e^{(\gamma +\widehat K)t}\in L^\infty(\RR)$. Also, since $\widehat K$ is analytic in the strip $\Im m(\xi )\in (-1, 5/2)$ it follows that  $\widehat f_0 e^{(\gamma +\widehat K)t}, \widehat f_0 \widehat K  e^{(\gamma +\widehat K)t}\in \mathscr S'(\RR)$ for all $f_0\in \mathscr S'(\RR)$. Their inverse Fourier transforms are then well-defined and 
\begin{align*}
f(t)&=\mathscr F^{-1}\left(\widehat e^{(\gamma +\widehat K)t\,f_0 } \right)\in \mathscr S'(\RR),\,\forall\ t>0
\end{align*}
while 
\begin{align*}
\frac {\partial f(t)} {\partial t}&=\mathscr F^{-1}\left(e^{(\gamma +\widehat K)t} \widehat f_0 (\gamma +\widehat K) \right)\in \mathscr S'(\RR),\,\forall\ t>0,
\end{align*} 
and equation  (\ref{prop1.e3}) is satisfied by taking the inverse Fourier transform in both sides of \eqref{Esol1}.
\end{proof}
 
\begin{cor}
\label{Cor2.2}
Suppose that \( f_0 \in L^p \) for some \( p \ge 1 \) and let \( f \) be the distribution in (\ref{prop1.e2}). Then, \( f \in C([0, \infty); L^p(\mathbb{R})) \cap 
C^1((0, \infty); L^p(\mathbb{R})) \) and
satisfies
\begin{align}
\label{Cor2.2E1}
\frac{\partial f(t)}{\partial t} = \gamma f(t) + K\ast f(t), \quad \forall t > 0
\end{align}
in \( L^p(\mathbb{R}) \). Moreover, for all \( t > 0 \),
\begin{align}
\label{Cor2.2E2}
\| f(t) \|_p \le e^{\gamma t} \| f_0 \|_p \left( 1 + \| K \|_1 \left( e^{\kappa \| K \|_1 t} - 1 \right) \right).
\end{align}
\end{cor}

\begin{proof} Since \eqref{prop1.e2} may also be written as
\begin{align*}
f(t, x)=e^{\gamma t}\left( f_0(x)+(H(t)\ast f_0 )(x)\right)
\end{align*}
and $H(t)\in L^1(\RR)$, it follows that $f(t)\in L^p(\RR)$ if  $f_0\in L^p(\RR)$.
Then, $f(t)\ast K \in L^p(\RR)$, since $K\in L^1(\RR)$. 
The right-hand side of \eqref{prop1.e3} is then in $L^p(\RR)$ and the Corollary follows using Proposition \ref{S1PH1} with $p=1$.
\end{proof}

\begin{prop}
\label{S3PP1}
For all  $f_0\in L^p (\RR),\, p\ge 1$  the function $f$ given in (\ref{prop1.e2}) is the unique solution of  (\ref{Cor2.2E1}) in $C^1((0, T); L^p(\RR)\cap C([0, T); L^p(\RR))$.
\end{prop} 
\begin{proof}
Suppose that $f_0=0$, and multiply the equation by $pf^{p-1}$,
\begin{align*}
\frac {\dd} {\dd t} \| f(t) \|^p_p & = 2\gamma p \| f(t) \|^p_ p+p\int  _{\RR }f^{p-1}(t, x) (f(t)\ast K)(x)\dd x\\
& \le 2\gamma p \| f(t) \|^p_ p+p \| f^{p-1} \|_{ p' } \| f(t)\ast K \|_p\\
& \le 2\gamma p \|f(t)\|^p_ p+p\left(\int  _{ \RR }|f(t, x)|^{p}dx\right)^{\frac{p-1}{p}} \| K \|_1 \|f(t) \|_p\\
& =2\gamma p \| f(t) \|^p_ p+p \| f(t)\|_p^p, 
\end{align*}
the fact that $K \in L^1(\mathbb{R})$. Then, $\| f(t) \|_p^p \le \| f(0) \|_p^pe^{(2\gamma +p)t}$ and $\| f(t) \|_p^p=0$.
\end{proof}

\begin{cor}
\label{S3C2}
If $f_0\in L^p(\RR)$ for some $p\ge 1$ is such that $f_0\ge0$, the solution  of  (\ref{Cor2.2E1}) given in (\ref{prop1.e2}) is such that $f(t)\ge 0$ for all $t>0$.
\end{cor}
\begin{proof}
The solution $f$ may also be obtained from a solution of
\begin{align*}
\frac {\partial g(t)} {\partial t}=K\ast g(t),\,\,\,g(t, x)=e^{-2\gamma  t}f(t, x).
\end{align*}
A solution of that equation may be obtained as a fixed point of
\begin{align*}
\mathscr T (g)(t, x)=g_0(x)+\int _0^t (K\ast g(s))(x)\dd s
\end{align*}
for $t$ small enough. Since $K\ge 0$, if $f_0\ge 0$ and $g(s)\ge 0$ for all $s \in [0,t]$ and thus the same is true for $\mathscr T (g(t))$.
\end{proof}

The next result solves  the Cauchy problem for equation (\ref{E8}) in weighted $L^p$ spaces.
\begin{prop} 
\label{S4P1}
Suppose that $\theta \in (-1, 5/2)$. Then, for all $p \ge 1$ and for all bounded intervals $I \subset [0, \infty)$, there exists a constant $C_{H,1} > 0$ such that
\begin{align*}
\|f_0 \ast H(t)\|_{L^p(e^{\theta x})} \le C_{H,1} \ell(t) \|f_0\|_{L^p(e^{\theta x})}, \quad \forall t \in I.
\end{align*}
\end{prop}

\begin{proof} 
The proof follows as in  \cite[page 104]{B}, by exploiting properties of convolutions. 
We first remind that from Corollary \ref{H2}, for all $\theta\in (-1, 5/2 )$
\begin{align*}
\int _{ \RR }|H(t, x)|e^{\theta x}\dd x\le \ell(t) \left(
\int _{ -\infty }^0e^{(\theta+1) x}(1+|x|^3)\dd x+
\int _0^\infty e^{(\theta-\frac {5} {2})x }(1+x^3) \dd x\right)
=C _{ H, 1 }\ell(t),
\end{align*}
from where we deduce that $H(t)\in L^1\left(e^{\theta x}\right)$. Then observe that if $f_0\in  L^1\left(e^{\theta x}\right)$,
\begin{align*}
&\int  _{ \RR }e^{\theta x}|(f_0\ast H(t))(x)|dx
=\|f_0\|_{L^1\left(e^{\theta x}\right)  } \| H(t) \|_{L^1\left(e^{\theta x}\right)  }.
\end{align*}
Also, 
if $f_0\in L^p(e^{\theta x})$ then by the previous case the function  $y\mapsto H(t, x-y) f^p_0(y)$ belongs to $L^1\left(e^{\theta x}\right)$, and 
since $H(t, y)^{1/p'}\in L^{p'}\left(e^{\theta x}\right)$, it follows from H\"older's inequality that 
\begin{align*}
|(H(t)\ast f_0)(x)|^p\le  \| H(t)\|_{ L^1(e^{\theta x}) }^{p/p'} \left( |H(t)| \ast |f_0|^p \right) (x).
\end{align*}
Finally, $|f_0|^p\in L^1\left(e^{\theta x}\right)$, and so we conclude as
\begin{align*}
\| H(t)\ast f_0 \|^p_{L^p\left(e^{\theta x}\right)}\le \|H(t)\| _{L^1\left(e^{\theta x}\right)}^{p/p'} \| H(t)\|_{ L^1(e^{\theta x}) } \| f_0 \|^p _{L^p\left(e^{\theta x}\right)}.
\end{align*}
\vskip -0.5cm
\end{proof}

The Cauchy problem for equation \eqref{E8} is also solved in the following spaces
\begin{align}
&X_{A, B} = \{ f \in L^\infty_{\text{loc}}(\mathbb{R}) \mid \| f \|_{X_{A, B}} < \infty \}, \\
&\| f \|_{X_{A, B}} = \sup_{x < 0} e^{-Ax} |f(x)| + \sup_{x > 0} e^{-Bx} |f(x)|.
\end{align}
\begin{cor}
\label{cor4.5}
Suppose that $f_0 \in L^p_{\text{loc}}(\mathbb{R})\cap X_{A, B}$ for some $p \geq 1$. \\
If $-\frac{5}{2} < A, B < 1$, $H(t) \ast f_0 \in X_{A, B}$ and there exists a constant $C = C(A, B)$ such that
\begin{align}
\label{cor4.5E2}
\| H(t) \ast f_0 \|_{X_{A, B}} \leq C \ell(t).
\end{align}
If $B < -\frac{5}{2}$ and $A > 1$, then there exists a constant $C > 0$ such that
\begin{align}
\sup_{x < 0} |H(t) \ast f_0(x)| (1 + |x|^3)e^{-x} + 
\sup_{x > 0} |H(t) \ast f_0(x)| (1 + |x|^3)e^{\frac{5x}{2}} < C \ell(t). 
\label{cor4.5E7}
\end{align}
If moreover $B < -\frac{1}{2} < A$, then
\begin{align}
\label{cor4.5E3}
\int_{\mathbb{R}} |(H(t) \ast f_0)(x)| e^{\frac{x}{2}} \, \dd x < \infty, \quad \forall t > 0.
\end{align}
\end{cor}

\begin{proof} By hypothesis $f_0\in L^p(\RR)$ and then, by Proposition \ref{S3PP1}, the function $H(t)\ast f_0$  is the unique solution of \eqref{Cor2.2E1} in $C^1((0, T); L^p(\RR)\cap C([0, T); L^p(\RR))$. 

We start with the case when $x>0$. As long as $\frac {5} {2} + A > 0$, $B-1<0$,
\begin{align}
\label{cor4.5E9}
&|(H(t)\ast f_0)(x)|\le C_0\ell(t)e^{-\frac {5} {2} x}\int _{-\infty }^0(1+(x-y)^3)e^{\left(\frac {5} {2} +A\right)y}\dd y\nonumber\\
&\hspace{3cm}+ C_0\ell(t)e^{-\frac {5} {2} x}\int _0^x e^{\left(\frac {5} {2} +B\right)y}(1+(x-y)^3) \dd y\nonumber \\ & 
\hspace{3.5cm}+C_0\ell(t)e^{x}\int_x^\infty  e^{(B-1 )y}(1+(x-y)^3)\dd y\nonumber\\
&\hspace{2.7cm} \le C'\ell(t) \left(e^{-\frac {5} {2} x}\left(1+x^3 \right)+ e^{Bx}+ e^{Bx}\right).
\end{align}
Similarly for $x<0$, as long as $\frac {5} {2} +A>0$, $B-1<0$,
\begin{align}
\label{cor4.5E11}
&|(H(t)\ast f_0)(x)|\le C_0\ell(t) e^{-\frac {5} {2} x}\int_{ -\infty }^xe^{\left(\frac {5} {2} +A\right)y}(1+(x-y)^3)\dd y \nonumber\\
&\hspace{3cm}+C_0\ell(t) e^{x}\int_x^0 e^{\left(A-1 \right)y}(1+(x-y)^3)\dd y\nonumber\\
& \hspace{3.5cm} + C_0\ell(t) e^{x}\int_0^\infty e^{(B-1)y}(1+(x-y)^3)\dd y\nonumber\\
&\hspace{2.7cm} \le C'\ell(t) \left(e^{Ax}+e^{Ax}+e^x(1+|x|^3)\right).
\end{align}
As $B<-5/2$, the term $e^{Bx}$ wins among all the terms for $x>0$ while as $A<1$, $e^{Ax}$ wins among all the terms for $x<0$. This concludes the first part of the Corollary. 
Property \eqref{cor4.5E7} follows as well from \eqref{cor4.5E9} and \eqref{cor4.5E11} depending on the hypothesis on $A$ and $B$.
 Property (\ref{cor4.5E3}) follows from,
\begin{align*}
\int  _{ \RR }|(H(t)\ast f_0)(x)|e^{\frac {x} {2}}dx\le C_0\ell(t) \int  _{ -\infty }^0 e^{\left(A+\frac {1} {2} \right)y}\dd y+
C_0\ell(t) \int_0^\infty e^{\left(B+\frac {1} {2} \right)y}\dd y
\end{align*}
where the two integrals in the right-hand side converge since $A+\frac {1} {2}>0$ and $B+\frac {1} {2}<0$.
\end{proof}

\begin{rem}
\label{remS}
The map $S(t): f_0\to H(t)\ast f_0$ belongs to $\mathscr L\left(L^p(\RR)\right)\cap \mathscr L\left(L^p(e^{\theta x})\right)$  for $p\ge 1$ 
 and  $S(t)\in \mathscr L(X _{ A, B })$ too when $A$ and $B$ satisfy $-5/2<A, B<1$. For all $t>0$,
\begin{align*}
&\| S(t)\| _{ \mathscr L\left(L^p(\RR)\right) }\le  e^{\gamma t}\left(1+ \| K\|_1\left(e^{\kappa \|K\|_1t}-1 \right)\right)\\
&\| S(t) \| _{\mathscr L\left(L^p(e^{\theta x})\right)  }\le e^{\gamma  t}\Big(1+C _{ H, 1 }\ell(t)\Big)\\
&\|S(t)\| _{  \mathscr L(X _{ A, B }) }\le e^{\gamma t}\Big(1+C(A, B)\ell(t)\Big).
\end{align*}
\end{rem}

Let us  show for further comparison, some properties of the maps $S(t)$ defined in Remark \ref{remS} acting on spaces $X _{ A, B }$.
\begin{cor}
\label{corass}
Suppose that, for some $A \in \left( -\frac{5}{2},1 \right)$, 
\begin{equation}
 f_0(x) = e^{Ax} \1_{x<0}. \label{corassE00}
 \end{equation}
Then, for all $a' \in (-1, -A)$, there exists a constant $C > 0$ such that
\begin{align}
\left\vert  f(t, x) - e^{t \Omega(-iA)} e^{Ax} 
\right\vert  \leq C t e^{t \gamma} e^{-a'x}, \quad \forall x < 0, \label{corassE1}
\end{align}
and for all $a'' \in \left(-A, \frac{5}{2} \right)$, there exists a constant $C' > 0$ such that
\begin{align}
\left\vert 
f(t, x) - e^{t \gamma} e^{Ax} 
\right\vert  \leq C t e^{t \gamma} e^{-a''x}, \quad \forall x > 0. \label{corassE2}
\end{align}
If, on the other hand, $f_0(x) = x^B \1 _{x>0}$ for some $B \in \left(-\frac{5}{2}, 1\right)$, then for all $b' \in \left(-B, \frac{5}{2}\right) $, there exists a constant $C > 0$ such that
\begin{align}
\left\vert 
f(t, x) - e^{t \Omega(-iB)} e^{Bx} 
\right\vert \leq C t e^{t \gamma} e^{-b'x}, \quad \forall x > 0, \label{corassE10}
\end{align}
and for all $b'' \in (-1, -B)$, there exists a constant $C' > 0$ such that
\begin{align}
\left\vert 
f(t, x) - e^{t \gamma} e^{Bx} 
\right\vert \leq C t e^{t \gamma} e^{-b''x}, \quad \forall x > 0. \label{corassE20}
\end{align}
\end{cor}

\begin{proof}
Suppose first that $f_0(x)=e^{Ax}\1 _{ x<0 }$ with $A\in (-5/2, 1)$. Then, for all $\xi \in \CC$ such that $\Im m (\xi )>-A$,
\begin{align}
\widehat f_0(\xi )
=\frac {1} {\sqrt{2\pi }}
\int _{ -\infty }^0 e^{-ix\xi } e^{Ax} \dd\xi 
=\frac {1} {\sqrt{2\pi }(A-i\xi )} \label{eq:init data order for xi<0}
\end{align}
and then, if $\mathscr C_a=\{\xi \in \CC;\,\,\Im m (\xi )=a\}$ with $a\in (-A, 5/2)$, using the fact that the solution is given by the inverse Fourier transform of $\widehat f_0(\xi )e^{t\Omega (\xi)}$ and inserting \eqref{eq:init data order for xi<0}, we get 
\begin{align}
f(t, x)&=\frac {1} {2\pi }\int _{ \mathscr C _{ a } }\frac {e^{ix\xi }e^{t\Omega (\xi )}} {A-i\xi }\dd\xi=
\frac {1} {2\pi }\int _{ \mathscr C _{ a } }\frac {e^{ix\xi }e^{t(\gamma +\widehat K(\xi ))}} {A-i\xi } \dd\xi \nonumber\\
&=e^{t\gamma }e^{Ax}+\frac {e^{t\gamma }} {2\pi }\int _{ \mathscr C _{ a } }\left( e^{t\widehat K(\xi )}-1\right)\frac {e^{ix\xi }} {A-i\xi }\dd \xi. \label{corassE7}
\end{align}
A contour deformation argument yields then for $a'\in (-1, -A)$,
\begin{align*}
f(t, x)&=e^{t\gamma }e^{Ax}+e^{t\gamma }\left(e^{t\widehat K(-iA)}-1 \right)e^{Ax}+\frac {e^{t\gamma }} {2\pi }\int _{ \mathscr C _{ a' } }\left( e^{t\widehat K(\xi )}-1\right)\frac {e^{ix\xi }} {A-i\xi }\dd \xi
\end{align*}
from where we get,
\begin{align*}
\left\vert f(t, x)-e^{t\Omega (-iA)}e^{Ax} \right\vert 
\le Ce^{t\gamma }e^{-a'x}
\int_{\RR}\frac {\left|e^{t\widehat K(u+ia' )}-1\right| \dd u} {((A+a')^2+u^2)^{1/2}}
\end{align*}
with $A+a'<0$. It follows from (\ref{pgcd.e0}) in Proposition \ref{ProphatK},
\begin{align*}
\int_{\RR}\frac {\left\vert e^{t\widehat K(u+ia' )}-1\right\vert \dd u} {((A+a')^2+u^2)^{1/2}}
\le Ct \int_{\RR}\frac {\left (u^2+a'^2\right)^{-1/2} \dd u} {((A+a')^2+u^2)^{1/2}}
\end{align*}
and (\ref{corassE1}) follows. If now $x>0$, a contour deformation in \eqref{corassE7} gives for $a''\in (-A, 5/2)$,
\begin{align*}
f(t, x)-e^{t\gamma }e^{Ax}= \frac {e^{t\gamma }} {2\pi }\int _{ \mathscr C _{ a'' } }\left( e^{t\widehat K(\xi )}-1\right)\frac {e^{ix\xi }} {A-i\xi }\dd\xi 
\end{align*}
with,
\begin{align*}
\left| \int _{ \mathscr C _{ a'' } }\left( e^{t\widehat K(\xi )}-1\right)\frac {e^{ix\xi }} {A-i\xi }\dd\xi\right|\le 
Ct e^{-a''x} \int  _{\RR}\frac {\left (u^2+a''^2\right)^{-1/2}\dd u} {((A+a'')^2+u^2)^{1/2}}
\end{align*}
and  (\ref{corassE2}) follows.
Similarly, if $f_0(x)=e^{Bx}\1_{ x>0 }$ with $B \in (-5/2, 1)$ then for all $\xi \in \CC$ such that $\Im m(\xi )<-B$,
\begin{align*}
\widehat f_0(\xi )=\frac {1} {\sqrt{2\pi }}\int_0^\infty e^{-ix\xi } e^{Bx} \dd\xi =-\frac {1} {\sqrt{2\pi }(B-i\xi )}
\end{align*}
and if $b'' \in (-1, -B)$,
\begin{align*}
f(t, x)&=-\frac {1} {2\pi }\int _{ \mathscr C _{ b^{''} } }\frac {e^{ix\xi }e^{t\Omega (\xi )}} {B-i\xi }\dd\xi=
-\frac {1} {2\pi }\int _{ \mathscr C _{ b^{''} } }\frac {e^{ix\xi }e^{t(\gamma +\widehat K(\xi ))}} {B-i\xi }\dd\xi \nonumber\\
&=e^{t\gamma }e^{Bx}-\frac {e^{t\gamma }} {2\pi }\int_{ \mathscr C _{ b^{''} } }
\left( e^{t\widehat K(\xi )}-1\right)\frac {e^{ix\xi }} {B-i\xi }\dd\xi.
\end{align*}
Then,
\begin{align*}
f(t, x)&-e^{t\gamma }e^{Bx}=-\frac {e^{t\gamma }} {2\pi }\int _{ \mathscr C _{ b^{''} } }
\left( e^{t\widehat K(\xi )}-1\right)\frac {e^{ix\xi }} {B-i\xi }\dd\xi
\end{align*}
and \eqref{corassE20} follows in the same way as \eqref{corassE2}. If $x>0$, then arguing as for \eqref{corassE1},
\begin{align*}
f(t, x)&=e^{t\gamma }e^{Bx}+
e^{t\gamma }\left(e^{t\widehat K(-iB)}-1 \right)e^{Bx} + \frac {e^{t\gamma }} {2\pi }\int _{ \mathscr C _{ b' } }\left( e^{t\widehat K(\xi )}-1\right)\frac {e^{ix\xi }} {B-i\xi }\dd\xi
\end{align*}
for any $b'\in (-B, 5/2)$ and (\ref{corassE10}) follows.										
\end{proof}
\begin{rem}
It is remarkable how different the properties of the map $S(t)$ obtained in Corollary \ref{corass} are with respect  to the similar map $T(t)$  for the linearization of equation (\ref{eq:NL WKEC10})  around the KZ solution $\omega ^{-7/6}$. As proved indeed  in \cite[Lemma 3.9]{EscMischVel08}, these  have the following regularizing effect at $-\infty$: for all initial data  (\ref{corassE00}) with $A\in \left(-3/2, -7/6\right]$, there is $\lambda (t)\in \RR$ such that  $T(t)f_0(x)\sim \lambda (t) e^{-7x/6}$ as $x\to -\infty$  . This is in contrast with Corollary \ref{corass}, where the  behavior $e^{Ax}$ of the initial as  $x\to -\infty$ is observed at all times.
\end{rem}

As proved in the next Proposition, for some solutions of equation \eqref{E8} the quantity \eqref{S1ETM} (wave action) is constant in time. This property follows from the fact that $W(1)=0$. In general, the zeros of the function $W$ correspond to the conserved quantities. As seen in the first section, Proposition \ref{prop: station solutions_zeros of W},  $W(2)=0$ makes $p^{-2}$ to be a stationary solution of \eqref{E1}. 

\begin{cor}
\label{S3Pconsl}
Suppose that $f_0 \in L^1(e^{rx})$ for some $r \in (-1, \frac{5}{2})$. Then,
\begin{align*}
\int_{\RR} f(t, x) e^{rx} \, \dd x = e^{\lambda_r t} \int_{\RR} f_0(x) e^{rx} \, \dd x,
\end{align*}
where $\lambda_r > 0$ if $r \in (-1, \frac{1}{2})$ or $r \in (1, \frac{5}{2})$, $\lambda_r < 0$ if $r \in (\frac{1}{2}, 1)$, and $\lambda_r = 0$ if $r = \frac{1}{2}$ or $r = 1$.
\end{cor}

\begin{proof}
For all $r\in (-1, 5/2)$ and $f_0\in L^1(e^{rx})$, by  Proposition \ref{S4P1}, $f(t)\in L^1(e^{r x})$. Then we write
\begin{align*}
\int_{ \RR }f(t, x)e^{rx}\dd x
&=\widehat f(t, ir)
=\widehat f_0(i r)\widehat G(t, ir)
=\widehat f_0(i r)e^{(\gamma +\widehat K(ir))t}\\
&=\widehat f_0(i r)e^{(\gamma +W(2r))t}=\int _{\RR }f_0(x)e^{rx}\dd x\, e^{(\gamma +W(2r))t}.
\end{align*}
The result follows with $\lambda _r=\gamma +W(2r)$ and the properties of the function $W$.
\end{proof}

\subsection{Long time behavior of some solutions to (\ref{E8}).}
\label{Slongtime}
Our next result shows that for some initial data at last,  a concentration of waves towards zero frequency. Unlike in the non linear equation, this process  takes infinite time, at an exponential rate that is easy to obtain explicitly. More precisely,
 
\begin{theo}
\label{S4P4.1}
Suppose that $f_0\in X _{ A, B }$, where $A$ and $B$ satisfy  (\ref{EAB}).
Then, for $\sigma >\frac {1} {2}$ close to $1/2$,
\begin{align}
\int  _{ -\infty }^Lf(t, x)e^{\frac {x} {2}}\dd x=
\int  _{ \RR}f(t, x)e^{\frac {x} {2}}\dd x+\int  _{ \Im m\xi =\sigma  }\widehat f_0(\xi )e^{(\gamma +\widehat K(\xi ))t}
\frac {e^{L\left(i\xi +\frac {1} {2} \right)}} {i\xi +\frac {1} {2}} \dd \xi, \label{EABC}
\end{align}
where all the integrals are absolutely convergent and for $t\to \infty$,
\begin{align}
\label{S4T1E1}
\left\vert \int  _{ \Im m\xi =\sigma  }\widehat f_0(\xi )e^{(\gamma +\widehat K(\xi ))t}\frac {e^{L\left(i\xi +\frac {1} {2} \right)}} {i\xi +\frac {1} {2}} \dd \xi \right\vert \le 
Ce^{-\nu t} \int  _{ \RR }|f_0(x)|e^{\frac {3x} {4}}\dd x
\end{align}
where $-\nu =(\gamma +\widehat K(3i/4))<0$.
\end{theo}
\begin{proof}
Since $f_0\in X _{ A, B }$ with $-5/2<A, B<1$,   $f(t) \in X _{ A, B }$ with $-5/2<A, B<1$ for all $t>0$ and the function $g(t)$ defined as $g(t, x)=f(t, x)e^{x/2}$ belongs to  $ L^1(\RR)$ for all $t>0$.  Since $f_0\in L^2$ then $f(t)\in L^2(\RR)$,
$$
f(t, x)=\int_{ \RR }e^{ix\xi }\widehat f(t, \xi )\dd\xi=\lim _{ R\to \infty }\int  _{ \RR }e^{ix\xi }\widehat f(t, \xi )\dd\xi ,\,\,\text{in}\,\,L^2(\RR),
$$
 and using that the function $x\to e^{\frac {x} {2}}$ belongs to  $L^2(-\infty, L)$ for all finite $L\in \RR$:
\begin{align*}
\int_{ -\infty }^Lf(t, x)e^{\frac {x} {2}} \dd x &=\frac {1} {\sqrt{2\pi }}\int  _{ -\infty }^L \int  _{ \RR }e^{ix\xi }\widehat f_0(\xi )e^{(\gamma +\widehat K(\xi ))t}d\xi e^{\frac {x} {2}}\dd x\\
&=\lim _{ R\to \infty }\frac {1} {\sqrt{2\pi }}\int_{ -\infty }^L \int  _{-R}^R e^{ix\xi }\widehat f_0(\xi )e^{(\gamma +\widehat K(\xi ))t}\dd\xi e^{\frac {x} {2}}\dd x.
\end{align*}
Fubini Theorem then gives,
\begin{align*}
\int  _{ -\infty }^Lf(t, x)e^{\frac {x} {2}}\dd x&
=\lim _{ R\to \infty  }\frac {1} {\sqrt{2\pi }}\int  _{ -R }^R \widehat f_0(\xi )e^{(\gamma  +\widehat K(\xi ))t} 
\int  _{ -\infty }^L e^{ix\xi } e^{\frac {x} {2}}\dd x \dd\xi\\
&=\lim _{ R\to \infty }
\frac {1} {\sqrt{2\pi }}\int  _{-R}^R\widehat f_0(\xi )e^{(\gamma +\widehat K(\xi ))t}\frac {e^{L\left(i\xi +\frac {1} {2} \right)}} {i\xi +\frac {1} {2}} \dd \xi. 
\end{align*}
Since $\widehat f_0 $ and the function $\xi \to (i\xi +\frac {1} {2})^{-1}$ belong to $L^2(\RR)$,
\begin{align}
\label{S4E10}
\lim _{ R\to \infty }\frac {1} {\sqrt{2\pi }}\int  _{-R}^R\widehat f_0(\xi )e^{(\gamma + \widehat K(\xi ))t}\frac {e^{L\left(i\xi +\frac {1} {2} \right)}} {i\xi +\frac {1} {2}} \dd \xi=\frac {1} {\sqrt{2\pi }}\int  _{\RR}\widehat f_0(\xi )e^{(\gamma +\widehat K(\xi ))t}\frac {e^{L\left(i\xi +\frac {1} {2} \right)}} {i\xi +\frac {1} {2}} \dd \xi,
\end{align}
where the integral in the right hand side of \eqref{S4E10} is absolutely convergent. 

Since $f(t)$ satisfies  \eqref{EAB} it follows that $\widehat f(t)$ is analytic on the strip $\Im m(\xi )\in (B, A)$ and then the function under the integral sign at the right hand side of (\ref{S4E10}) is analytic in the strip $\Im m(\xi )\in (B , \frac {1} {2})$ and has a pole at $\xi =i/2$.
Moreover, for all $\xi =u+i\sigma $ such that $u\in \RR$,  $\sigma \in  (-\varepsilon , \frac {1} {2}+\varepsilon )$ with $\varepsilon >0$ small,
\begin{align*}
\left|\widehat f_0(\xi )e^{(\gamma +\widehat K(\xi ))t}\frac {e^{L(i\xi +\frac {1} {2})}} {i\xi +\frac {1} {2}}\right|
& \le Ce^{C't} |\widehat f_0(u+i\sigma )|
\frac {e^{-L\sigma }} {|-\sigma +\frac {1} {2}+iu|}\\
&\le Ce^{C't} |\widehat g_0(u)|\frac {e^{-L\sigma }} {|-\sigma +\frac {1} {2}+iu|}, 
\end{align*}
where $g_0(u)=e^{\sigma u}f_0(u)$ is such that,
\begin{align*}
\int  _{ \RR }|g_0(u)|\dd u&=\int  _{ \RR }|f_0(u)|e^{\sigma u}\dd u\\
&\le C \int_{ -\infty }^0 e^{(A+\sigma )u}\dd u+C\int _0^\infty e^{(B+\sigma)u }\dd u<\infty
\end{align*}
since $A+\sigma >A-\varepsilon >0$ and $B+\sigma < B+\frac {1}{2}+\varepsilon <0$. Therefore,
\begin{align*}
\left|\widehat f_0(\xi )e^{(\gamma +\widehat K(\xi ))t}\frac {e^{L(i\xi +\frac {1} {2})}} {i\xi +\frac {1} {2}}\right|
\le  Ce^{C't} \| g_0\|_ 1|\frac {e^{-L\sigma }} {|-\sigma +\frac {1} {2}+iu|}\to 0,\,\,|u|\to \infty.
\end{align*}
Classical deformation of integration contour gives then,
\begin{align*}
\int_{\RR}\widehat f_0(\xi )e^{(\gamma +\widehat K(\xi ))t}\frac {e^{L\left(i\xi +\frac {1} {2} \right)}} {i\xi +\frac {1} {2}} \dd \xi&=
(-i) 2\pi  i \widehat f_0(i/2)e^{\left(\gamma +\widehat K(i/2) \right)t}\\
& \hspace{1cm}+\int  _{ \Im m(\xi )=\sigma  }\widehat f_0(\xi )e^{(\gamma +\widehat K(\xi ))t}\frac {e^{L\left(i\xi +\frac {1} {2} \right)}} {i\xi +\frac {1} {2}} \dd \xi.
\end{align*}
Since $\gamma +\widehat K(i/2)=0$,
\begin{align}
\label{S4ECD1}
\int  _{ -\infty }^Lf(t, x)e^{\frac {x} {2}}dx= \sqrt {2\pi }\widehat f_0(i/2)+
\frac {1} {\sqrt{2\pi }}\int  _{ \Im m\xi =\sigma  }\widehat f_0(\xi )e^{(\gamma +\widehat K(\xi ))t}\frac {e^{L\left(i\xi +\frac {1} {2} \right)}} {i\xi +\frac {1} {2}} \dd\xi 
\end{align}
and \eqref{EABC} follows.

The function under the integral sign at the right hand side of \eqref{S4ECD1} is now analytic on the domain $\Im m(\xi )\in (i/2, 5i/2)$. Moreover, 
 since  $f_0\in X _{ A, B }$, $-5/2<A, B<1$, the function $h_\sigma $ defined as  $h_\sigma (x)=e^{-\sigma x}f_0(x)$ is such that $h_\sigma \in L^2(\RR)$,  and then $\hat h_\sigma \in L^2(\RR)$, for all $\sigma \in (B, A)$. Since for all $u\in \RR, \sigma \in \RR$,
$$ 
\mathscr F\left( h_\sigma \right)(u)=\frac {1} {\sqrt{2\pi }}\int  _{ \RR }f_0(x)e^{-\sigma x}e^{ixu}\dd x=\frac {1} {\sqrt{2\pi }}\int_{ \RR }f_0(x)e^{ix(u+i\sigma )}\dd x=\hat f_0(u+i\sigma )
$$
it follows, for all $\sigma \in (1/2, A)$,
\begin{align*}
\int_{ \Im m\xi =\sigma  }\left|\widehat f_0(\xi )e^{(\gamma +\widehat K(\xi ))t}\frac {e^{L\left(i\xi +\frac {1} {2} \right)}} {i\xi +\frac {1} {2}}\right| \dd |\xi|
\le \int  _{ \RR }|\hat h_\sigma (u)|\frac {\dd u} {|iu-\sigma +\frac {1} {2}|}\le C \| h_\sigma \|_2
\end{align*}
and the integral 
\begin{align*}
\frac {1} {\sqrt{2\pi }}\int  _{ \Im m\xi =\sigma  }\widehat f_0(\xi )e^{(\gamma +\widehat K(\xi ))t}\frac {e^{L\left(i\xi +\frac {1} {2} \right)}} {i\xi +\frac {1} {2}} \dd \xi
\end{align*}
is absolutely convergent for all $\sigma \in (1/2, A)$.

The same argument shows that the contour deformation argument may still be applied to the second integral at the right hand side of (\ref{S4ECD1}) as long as the contour does not leave the domain where $\Im m(\xi )\in (1/2, A)$.

The critical point of $\Omega (\xi )=(\gamma +\widehat K(\xi ))$ is $\xi =3i/4$. Since $\Omega $ is analytic on the strip $\Im m\xi \in (-1, 5/2)$,
\begin{align*}
\Omega (\xi )=\Omega (3i/4)+\frac {(\xi -3i/4)^2} {2}\Omega ''(3i/4)+\mathcal O\left(\xi -3i/4 \right)^3,\,\,\xi \to 3i/4,
\end{align*}
where, use of Mathematica gives $\Omega ''(3i/4)=-4W''(3/2)<0$. Since the function $\Omega $ is analytic in the strip $\Im m\xi \in (-1, 5/2)$ the contour of integration $\Im m(\xi )=\sigma $ may now be deformed to the line
\begin{align*}
\mathcal C'=\left\{\xi = u+\frac {3i} {4}, u\in \RR \right\}
\end{align*}
 along which the following hold:
 
 \begin{itemize}
    \item[(i)] $\Omega(\xi) = \Omega(3i/4) + \frac{1}{2}(\xi - 3i/4)^2 \Omega''(3i/4) + \mathcal{O}((\xi - 3i/4)^3),\ \xi \to 3i/4$
    \item[(ii)] $\Omega(u + 3i/4) \in \RR$ for all $u \in \RR,$ and 
    \begin{align*}
        \forall \rho > 0,\, \exists \mu > 0;\  \Omega(u + 3i/4) \le \Omega(3i/4) - \mu,\,\, \forall |u| \ge \rho \,\, \text{(cf.\ Figure 1)}.
    \end{align*}
\end{itemize}

\begin{figure}[]
\centering
\includegraphics[width=0.7\textwidth,height=0.3\textheight]{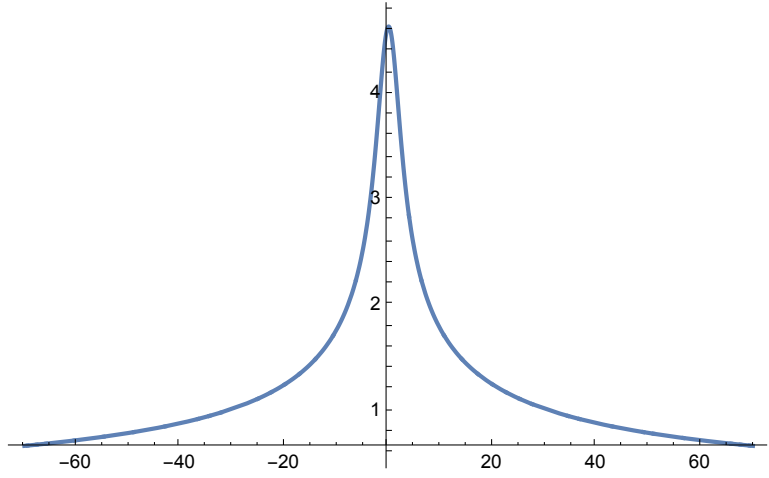}
\caption{The function $W\left(\frac {3} {2}+iv\right)$ for $v\in (-70, 70)$.}
\end{figure}
Then from (i), for $R_0$ small enough
\begin{align*}
\Omega (\xi )=\Omega (3i/4)+\frac {u^2} {2}\Omega ''(3i/4)+\mathcal O(u)^3,\,\forall \xi \in \mathcal C',\, |u|\le R_0.
\end{align*}
while from (ii), there exists $\mu >0$ such that
\begin{align*}
 \Omega (\xi )\le \Omega (3i/4)-\mu,\,\,\forall \xi \in \mathcal C',\,\,|u|>R_0.
\end{align*}
The curve $\mathcal C'$ may then be decomposed into $\mathcal C'=\mathcal C_1'+\mathcal C_2'$, where 
\begin{align*}
\mathcal C'_1=\{\xi \in \mathcal C';\,|u|>R_0\}, \quad \mathcal C'_2=\{\xi \in \mathcal C';\,|u|<R_0\}, 
\end{align*}
so that 
\begin{align*}
\int_{ \mathcal C' }\widehat f_0(\xi )e^{ t\Omega (\xi )}
\frac {e^{L\left(i\xi +\frac {1} {2} \right)}} {i\xi +\frac {1} {2}} \dd\xi
 &=\int_{ \mathcal C'_1}\widehat f_0(\xi )e^{t\Omega (\xi )}\frac {e^{L\left(i\xi +\frac {1} {2} \right)}} {i\xi +\frac {1} {2}} \dd\xi +\int_{\mathcal C'_2}\widehat f_0(\xi )e^{t\Omega (\xi )}\frac {e^{L\left(i\xi +\frac {1} {2} \right)}} {i\xi +\frac {1} {2}} \dd\xi
 \\&=: I_1+I_2.
\end{align*}
In the first integral, $ I_1$
\begin{align*}
\left|e^{t\Omega (\xi )}\frac {e^{L\left(i\xi +\frac {1} {2} \right)}} {i\xi +\frac {1} {2}} \right|\le \frac {Ce^{(\Omega (3i/4)-\mu )t}} {\left|i\xi +\frac {1} {2}\right|}
\end{align*}
using the conclusion from point (ii). In total, 
\begin{align*}
\left|I_1\right|\le Ce^{(\Omega (3i/4)-\mu )t} \int  _{ \mathcal C'}\left|\hat f_0(\xi )\right|\frac {d\xi } {\left|i\xi +\frac {1} {2}\right|}
\le Ce^{(\Omega (3i/4)-\mu )t} \| h_\sigma \|_2.
\end{align*}
after C-S in the last inequality.
In the second integral, for $R_0$ small,
\begin{align*}
I_2=e^{t\Omega (3i/4)+\mathcal O(R_0)^3}\int  _{|u|<R_0}\widehat f_0(u+3i/4 )e^{  \frac {t} {2} u^2 \Omega ''(3i/4)}\frac {e^{L\left( -\frac {1} {4}+iu \right)}} { -\frac {1} {4}+iu} \dd u, 
\end{align*}
we perform the change of variables, $u\sqrt {-t\Omega ''(3i/4)}=2z$ which gives,
\begin{align*}
&\int  _{|u|<R_0}\widehat f_0(u+3i/4 )e^{  \frac {t} {2} u^2
 \Omega ''(3i/4)}\frac {e^{L\left( -\frac {1} {4}+iu \right)}} { -\frac {1} {4}+iu} \dd u=\\
&\frac {2} {\sqrt {-t\Omega ''(3i/4)}}\int\limits _{2|z|\le R_0\sqrt {-t\Omega ''(3i/4)}}\widehat f_0\left(\frac {2z} {\sqrt {-t\Omega ''(3i/4)}}+\frac {3i} {4} \right)e^{-z^2}
j(z)\dd z
\end{align*}
where 
\begin{align*}
j(z)&=\left(\frac {2iz} {\sqrt {-t\Omega ''(3i/4)}}-\frac {1} {4} \right)^{-1}\exp\left(\frac {2iLz} {\sqrt {-t\Omega ''(3i/4)}}-\frac {L} {4} \right)\\
&=-4e^{-\frac {L} {4}}\left(1+\mathcal O\left( R_0\right) \right),\, R_0\to 0.
\end{align*}
By continuity of $f_0$,
\begin{align*}
\widehat f_0\left(\frac {2z} {\sqrt {-t\Omega ''(3i/4)}}+\frac {3i} {4} \right)=\widehat f_0\left(\frac {3i} {4}\right)(1+o(1)),\,\,t\to \infty
\end{align*}
and then, as $t\to \infty$ and $R_0\to 0$
\begin{align*}
I_2=\frac {-2e^{t\Omega (3i/4)}} {\sqrt {-t\Omega ''(3i/4)}}\left(1+\mathcal O(R_0)^3 \right) \widehat f_0\left(\frac {3i} {4}\right)(1+o(1))4e^{-\frac {L} {4}}\left(1+\mathcal O\left( R_0\right) \right)
\end{align*}
from where (\ref{S4T1E1}) follows.
\end{proof}

The second and final result of this section pertains to initial data $f_0$ whose Fourier transform exhibits good properties of analyticity and integrability. Under these conditions, the pointwise behaviour of $f(t, x)$ as $|x|/t\to \infty$ and $|x|/t\to 0$ may be described. It is then possible to observe the growth or decay of the solution for $x\lessgtr (4/3)\Omega (3i/4)t$, where $\Omega (3i/4)\approx -0.1572 $.
\begin{prop}
\label{S4P2}
Suppose that $f_0$ is such that $\hat f_0\in X _{ A, B }$ with $-5/2<A, B<1$ and the function $g_0$ defined as $g_0(x)=f_0(x)e^{\frac {x s_*} {2}}$ for $x\in \RR$ satisfies  $\widehat g_0 \in L^1(\RR)$.
Then, 
\begin{align}
f(t, x)\sim
\begin{cases}
\displaystyle{\frac {\widehat f_0(-i)} {\sqrt{3\pi }}\frac {e^{x}t^{\frac {1} {6}}} {|x|^{\frac {2} {3}}}e^{3t^{\frac {1} {3}}|x|^{\frac {2} {3}}},\,\frac {x} {t}\to -\infty,}\\
\displaystyle{\frac {\widehat f_0(5i/2)} {\sqrt{3\pi }}\frac {e^{-\frac {5x} {2}}t^{\frac {1} {6}}} {|x|^{\frac {2} {3}}}e^{3t^{\frac {1} {3}}|x|^{\frac {2} {3}}},\,\frac {x} {t} \to \infty.}
\end{cases}
\end{align}
and,
\begin{align}
f(t, x)\sim
 \frac {\widehat f_0(3i/4)e^{-3x/4}e^{\Omega (3i/4)t}} {\sqrt{2\pi \Omega ''(3i/4)t}},\,\,\frac {|x|} {t}\to 0.
 \end{align}
\end{prop}
\begin{proof}
As in the Theorem \ref{S4P4.1}, since $f_0\in X _{ A, B }$ with $-5/2<A, B<1$,   $f(t) \in X _{ A, B }$ with $-5/2<A, B<1$ for all $t>0$.The proof follows now from steepest descent method applied to the integral
\begin{align*}
f(t, x)=\frac {1} {\sqrt{2\pi }}\int  _{ \RR }e^{ix\xi +(\gamma +\widehat K(\xi ))t}\widehat f_0(\xi )\dd\xi.
\end{align*}
However, the proof  is simpler to write in the variables $A, p$, using then the  Mellin transform and its inverse.
Write instead (notice the difference  of $p=e^{x/2},\, x=2\log p$ with the notation in \cite{MM}),
\begin{align*}
f(t, x)=\frac {1} {2\pi  i}\int  _{ \sigma -i\infty} ^{\sigma +i\infty} \Phi(s) e^{-\frac {x s} {2}+tW(s)}\dd s
\end{align*}
where $\Phi (s)$ is the Mellin transform of the initial data $A_0$. Notice first that the unique critical point $s_*$  of the function 
$\psi (s; t, x)=-\frac {x s} {2}+tW(s)$ is such that,
\begin{align*}
W'(s_*)=-\frac {x} {2t}.
\end{align*}
Since $W$ is analytic for  $\Re e(s)\in (-2, 5)$ with only two poles at $s=-2, s=5$ in $\Re e(s)\in [-2, 5]$ and satisfies,
\begin{align*}
W'(s)&= -\frac {8} {(s+2)^3}-\frac {8} {3(s+2)^2}+\mathcal O(1),\,\,s\in (-2, 5),\, s\to -2\\
W'(s)&=-\frac {8} {(s-5)^3}+\frac {8} {3(s-5)^2}+\mathcal O(1),\,\,s\in (-2, 5),\,\, s\to 5
\end{align*}
two cases may  arise as $xt^{-1}\to \pm \infty$. Suppose first $x\to -\infty$, and then $W'(s_*)\to \infty$.  If $\Re e(s_*)\in (-2, 5)$, then
\begin{align*}
-\frac {x} {2t}=W'(s_*)=-\frac {8} {(s_*+2)^3}+\mathcal O\left(s_*+2 \right)^{-2}\to \infty,\,\,\,\, \frac {x} {t}\to -\infty,
\end{align*}
and since, for $|x/t|$ large enough,
\begin{align*}
\frac {x} {2t}=\left| -\frac {8} {(s_*+2)^3}+\mathcal O\left(s_*+2 \right)^{-2}\right|\le C\left|\frac {8} {(s_*+2)^3} \right|
\end{align*}
it follows that $s_*+2\to 0$ as $x/t\to -\infty$ and,
\begin{align*}
s_*+2 = 
\mathcal O\left( \frac {t} {x}\right)^{1/3},\,\,\,\frac {x} {t}\to -\infty.
\end{align*}
Then,
\begin{align*}
&\frac {1} {-\frac {8} {(s_*+2)^3}+\mathcal O\left(s_*+2 \right)^{-2}}=-\frac {2t} {x}
\end{align*}
meaning that 
\begin{align*}
&(s_*+2)^3 =\frac {16t} {x}+\mathcal O(s_*+2)^4
=\frac {16 t} {x}+\mathcal O\left( \frac {t} {x}\right)^{4/3}
\end{align*}
and so 
\begin{align*}
& s_*=-2+2\left(\frac {2 t} {x}\right)^{1/3}+\mathcal O\left( \frac {t} {x}\right)^{4/9}
\end{align*}
and the point $s _*$ is then in the domain $\Re e(s)\in (-2, 5)$ for $x/t\to -\infty$. The integration contour  $\Re e(s)=\sigma $ must now be deformed  to the  contour $\mathscr C=\left\{s= s_*+iv,\,v\in \RR\right\}$, along which the following hold,

\begin{itemize}
\item[(i)]$\psi (s; t, x)=\psi (s_*; t, x)+\frac {1} {2}(s-s_*)^2 tW''(s_*)+\mathcal O(s-s_*)^3,\,\,\frac {x} {t}\to -\infty $ 
\item[(ii)]$ \forall \rho >0,\,\exists \mu >0;\,\,\Re e(W(s_*+i v))\le \Re e(W (s_*))-\mu ,\,\,\forall |v|\ge \rho.$
\end{itemize}
\begin{figure}[]
\centering
\includegraphics[width=0.7\textwidth,height=0.3\textheight]{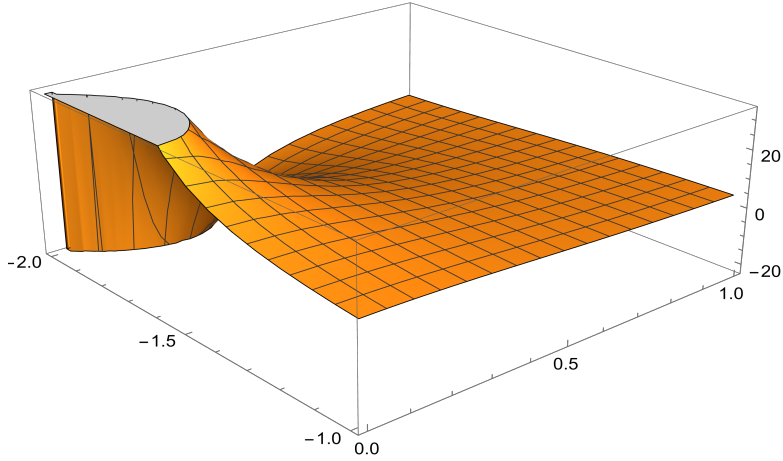}
\caption{The function $Re[W\left(u+iv\right)]$ for $u\in (-2, -1), v\in (0, 1)$.}
\end{figure}
The curve $\mathscr C$ may then be written as $\mathscr C=\mathscr C_1\cup \mathscr C_2$,\\
\begin{align*}
\mathscr C_1=\left\{s= s_*+iv,\,|v|>\rho \right\}\\
\mathscr C_2=\left\{s= s_*+iv,\,|v|<\rho \right\}
\end{align*}
and
\begin{align*}
\int  _{\mathscr C} \Phi(s) e^{\psi (s; t, x)}\dd s=
\int  _{\mathscr C_1} \Phi(s) e^{\psi (s; t, x)}\dd s+
\int  _{\mathscr C_2} \Phi(s) e^{\psi (s; t, x)}\dd s=: I_1+I_2.
\end{align*}
In the term $I_1$, point (ii) gives $\Re e(W(s_*+i v))\le \Re e(W (s_*))-\mu$, and then,
\begin{align*}
\left| \Phi (s) e^{\psi (s; t, x)}\right|=\left| \Phi (s)\right| e^{\Re e(\psi (s; t, x))}\le  \left| \Phi (s)\right| e^{-\frac {xs_*} {2}+t(\Re e(W (s_*))-\mu)}
\end{align*}
so that 
\begin{align*}
\left|\int_{\mathscr C_1} \Phi(s) e^{\psi (s; t, x)}\dd s\right| 
&\le 
C\int  _{ |v|>\rho  } \left| \Phi (s_*+iv)\right| e^{-\frac {xs_*} {2}+t(\Re e(W (s_*))-\mu)}\\
&\le Ce^{-\frac {xs_*} {2}+t(\Re e(W (s_*))-\mu)}\int_{ |v|>\rho  }\left| \Phi (s_*+iv)\right|\dd v.
\end{align*}
By hypothesis,
\begin{align*}
\Phi (s_*+iv)&=\int _0^\infty p^{s_*+iv}A(p)\frac {\dd p} {p}=\int  _{ \RR }e^{\frac {x(s_*+i v))} {2}}f_0(x)\dd x\\
&=\int_{ \RR }e^{\frac {xi v} {2}}e^{\frac {x s_*} {2}}f_0(x)\dd x=\sqrt {2\pi }\, \widehat g_0(v/2)\in L^1(\RR).
\end{align*}
where $g_0(x)=f_0(x)e^{\frac {x s_*} {2}}$.
On the other hand, if  $s=s_*+iv \in \mathscr C$ then $s-s_*=iv$ for some $v\in (-\rho , \rho )$ and then
\begin{align*}
I_2&=e^{\psi (s_*; t, x)}\int  _{|v|<\rho } \Phi(s_*+iv) e^{-\frac {1} {2}v ^2 tW''(s_*)+\mathcal O(v)^3}\dd v\\
&=e^{\psi (s_*; t, x)}\int  _{|v|<\rho } \Phi(s_*+iv) e^{-\frac {1} {2}v ^2 tW''(s_*)}\left(1+\left(e^{\mathcal O(v)^3}-1\right) \right)\dd v\\
&=e^{\psi (s_*; t, x)}\int  _{|v|<\rho } \Phi(s_*+iv) e^{-\frac {1} {2}v ^2 tW''(s_*)}\left(1+\mathcal O(\rho )^3  \right)\dd v
\end{align*}
where $W''(s_*)>0$.  
Under the change of variables $z=\sqrt {tW''(s_*)}\, v$, 
\begin{align*}
I_2&=\frac {e^{\psi (s_*; t, x)}} {\sqrt {tW''(s_*)}}
 \int\limits_{|z|<\rho  \sqrt {tW''(s_*)}}  \Phi\left(s_*+\mathcal O(\rho )\right)e^{-\frac {1} {2} z^2}
\left( 1+\mathcal O(\rho )\right)\dd z\\
&=\frac {e^{\psi (s_*; t, x)}} {\sqrt {tW''(s_*)}}\Phi\left(s_*+\mathcal O(\rho )\right)\left( 1+\mathcal O(\rho )\right) \int\limits_{|z|<\rho  \sqrt {tW''(s_*)}}  e^{-\frac {1} {2} z^2} \dd z\\
&=\sqrt{\frac {\pi } {2}}\frac {e^{\psi (s_*; t, x)}} {\sqrt {tW''(s_*)}}\Phi\left(s_*+\mathcal O(\rho )\right)\left( 1+\mathcal O(\rho )\right)
\left(1+\mathcal O\left( e^{-t\rho ^2W''(s_*)}\right) \right).
\end{align*} 
Since 
\begin{align*}
&s_*=-2+2\left(\frac {2  t} {x}\right)^{1/3}+\mathcal O\left( \frac {t} {x}\right)^{4/9}
\end{align*} 
we compute
\begin{align*}
\psi (s_*; t, x)&=-\frac {x s_*} {2}+tW(s_*)= 
x \left(1- \left(\frac {2\,  t} {x}\right)^{1/3}+\mathcal O\left( \frac {t} {x}\right)^{4/9} \right)\\
&\hspace{7cm} +\frac {4 t} {(s_*+2)^2}+t\mathcal O\left(s_* +2\right)^{-1}\\
&= x\left(1-\left(\frac {2t} {x}\right)^{1/3}+\mathcal O\left( \frac {t} {x}\right)^{4/9}\right)+4t\left(\frac {x}{16\,  t}\right)^{2/3}\left(1+\mathcal O\left(\frac {t} {x} \right)^{1/9} \right)+t\,\mathcal O\left(\frac {x} {t} \right)^{1/3}\\
&=x-(2t)^{1/3}x^{\frac {2} {3}}+t^{\frac {1} {3}}\left(\frac {x} {2}\right)^{\frac {2} {3}}+x\mathcal O\left(\frac {t} {x} \right)^{4/9}+t \mathcal O\left(\frac {x} {t} \right)^{\frac {5} {9}}+t\,\mathcal O\left(\frac {x} {t} \right)^{1/3}\\
&=x-\left( 2^{1/3}-2^{-2/3}\right)t^{\frac {1} {3}}x^{2/3}+t \mathcal O\left(\frac {x} {t} \right)^{\frac {5} {9}}.
\end{align*}
Also, 
\begin{align*}
W''(s_*)=\frac {24 } {(s_*+2)^4}+\mathcal O\left(s_*+2 \right)^{-3}
\end{align*}
which gives 
\begin{align*}
\frac {1} {\sqrt{tW''(s_*)}}&=\frac {(s_*+2)^2} {2\sqrt {6 t} }\left(1+\mathcal O\left(s_*+2 \right) \right)\\
&=\frac {1} {2\sqrt{6t}}\left(\frac {16 t} {x} \right)^{2/3}\left(1+\mathcal O\left(\frac {t} {x} \right) ^{1/9}\right)+\frac {1} {\sqrt t}\mathcal O\left( \frac {t} {x}\right)^{1/3}\\
&=\frac {2\times  2^{1/6}} {\sqrt 3} t^{1/6} x ^{-2/3}+\frac {1} {\sqrt t}\mathcal O\left( \frac {t} {x}\right)^{1/3}.
\end{align*}
Inserting these into $I_2$, gives the stated behaviour of $f$ as $x\to -\infty$. Analogously we conclude the behaviour at $x\to \infty$, as $s_*\to 5$. 
\end{proof}
\begin{rem}
If the initial data $f_0$ satisfies $\widehat g_0 \in L^1(\RR)$ for  $g_0(x)=f_0(x)e^{\frac {x s_*} {2}}$, then $g$ defined as $g(t, x)=
f(t, x)e^{\frac {x s_*} {2}}$ for $t>0, x\in \RR$ is such that $\widehat g(t) \in L^1(\RR)$ for all $t>0$ as it immediately follows from the following observation:
\begin{align*}
\widehat g(t, \xi )&=\widehat f\left(t, \xi +\frac {s_*} {2}\right)=\widehat G\left(t, \xi +\frac {s_*} {2}\right)\widehat f_0\left(\xi +\frac {s_*} {2}\right)\\
&=\widehat G\left(t, \xi +\frac {s_*} {2}\right)\widehat g_0\left(\xi\right)\in L^1(\RR),\,\,\text{since}\,\,\widehat G(t)\in L^\infty(\RR),\,\forall t>0.
\end{align*}
\end{rem}

\subsection{Back to the variables $A$ and $p>0$.}

In this subsection, we translate our obtained results in Fourier variables into the original variables $ A $ and $ p $. 
They can be summarized as follows.

\begin{itemize}
    \item[1.] Well-posedness:
    Equation \eqref{E1} is well-posed in the spaces:
    \begin{align}
    &(i)\quad L^q\left((0, \infty); \frac{\dd p}{p}\right), \quad \forall q \ge 1, \\
    &(ii)\quad L^q\left((0, \infty); p^{2\theta-1} \dd p \right), \quad q \ge 1, \quad \theta \in (-1, 5/2), \\
    &(iii)\quad \mathscr{Y}_{A, B} = \{g \in L^\infty_{\text{loc}}(0, \infty); \, ||g||_{\mathscr{Y}} < \infty\}, \\
    &\qquad ||g||_{\mathscr{Y}} = \sup_{0 < p < 1} (p^{2A} |g(p)|) + \sup_{p > 1} (p^{2B} |g(p)|), \nonumber
    \end{align}
    where $A, B$ satisfy \eqref{EAB}.

    \item[2.] Non-negative solutions:
    \begin{equation}
    A_0 \in L^q\left((0, \infty); \frac{\dd p}{p}\right) \ \& \ A_0 \ge 0 \Longrightarrow A(t) \ge 0, \quad \forall t > 0.
    \end{equation}

    \item[3.] Conservation of the wave action:
    \begin{align}
    A_0 \in L^1(0, \infty) \Longrightarrow \int_0^\infty A(t, p) \dd p = \int_0^\infty A_0(p) \dd p,
    \end{align}
    and growth of the energy:
    \begin{align}
    A_0 \in L^1((0, \infty); p^2 \dd p) \Longrightarrow \int_0^\infty A(t, p)p^2 \dd p = e^{W(3)t} \int_0^\infty A_0(p)p^2 \dd p.
    \end{align}

    \item[4.] Long time behaviour: If $A_0 \in \mathscr{Y}_{A, B} $ and $ A, B $ satisfy \eqref{EAB}, then
    \begin{align}
    \int_0^R A(t, p) \dd p = \int_0^\infty A_0(p) \dd p + C e^{-\nu t} \int_0^\infty |A_0(p)| p^{-1/4} \dd p.
    \end{align}
\end{itemize}

\begin{rem}
If $A_0\in \mathscr Y _{ A, B }$ and $A, B$ satisfy \eqref{EAB} it follows from properties 3) and 4),
\begin{align*}
\lim _{ t\to \infty }\int _R^\infty A(t, p) \dd p=0.
\end{align*}
If in addition $A_0\ge 0$, then for all $\varphi \in C([0, \infty))\cap L^\infty((0, \infty))$,
\begin{align*}
&\lim _{ t\to \infty }\int _0^\infty A(t, p)\varphi (p) \dd p= \varphi (0) m(A), \quad m(A)=\int _0^\infty A_0(p) \dd p.
\end{align*}
Indeed first, $A(t)\ge 0$ for all $t>0$.  By the continuity of $\varphi $, for all $\varepsilon >0$ and  $\rho$ small enough
\begin{align*}
&\int _0^\infty A(t, p)\varphi (p) \dd p-\int _0^\infty A_0(p) \dd p\, \varphi (0)=\\
&=\int _0^\rho A(t, p)\varphi (p) \dd p+\int _\rho ^\infty A(t, p)\varphi (p) \dd p-\int _0^\infty A(t, p)\varphi (0) \dd p\\
&=\int _0^\rho A(t, p)(\varphi (p)-\varphi (0)) \dd p+\int _\rho ^\infty A(t, p)\varphi (p) \dd p-\int _\rho ^\infty A(t, p)\varphi (0) \dd p.
\end{align*}
We have 
\begin{align*}
&\left|\int _0^\rho A(t, p)(\varphi (p)-\varphi (0)) \dd p\right|\le \varepsilon \|A_0\|_1,\\
&\left|\int _\rho ^\infty A(t, p)\varphi (p) \dd p \right|\le
 \| \varphi \| _{ \infty }\int _\rho ^\infty A(t, p)\dd p\to 0,\,\,t\to \infty,\\
&\left|\int _\rho ^\infty A(t, p)\varphi (0) \dd p\right|\le |\varphi (0)|\int _\rho ^\infty A(t, p) \dd p\to 0,\,\,t\to \infty.
\end{align*}
\end{rem}
\begin{rem}
By  the previous Remark and  (ii) of Proposition (\ref{S3Pconsl}), for all $R>0$,
\begin{align*}
&\lim _{ t\to \infty }\int _0^RA(t, p)p^2dp=0
\end{align*}
and then, for $\varepsilon >0$ as small as desired there exists  $T$ large enough, depending on $R$ and $\varepsilon $, such that
\begin{align*}
\int _R^\infty A(t, p)p^2\dd p\ge \frac {1} {1+\varepsilon }\int _0^\infty A(t, p)p^2 \dd p,\,\,\forall t\ge T.
\end{align*}
\end{rem}
\section{Appendix}
\label{Sapp}
\subsection{Results of \cite{MM}}
We record here some facts on the function $\mathcal{V}$ and its Mellin transform, that have been proved in \cite{MM}.

\begin{align}
&\mathcal V(p)=R(p)+S(p)\\
&R(p)=\left(\frac {2\log p} {1+p^2} -\frac {2} {p\sqrt{1+p^2}}\text{arctanh} \left(\frac {p} {\sqrt{1+p^2}} \right)\right)\Theta(1-p)-\nonumber\\
&-\left(\frac {2\log p} {p(1+p^2)} +\frac {2} {p\sqrt{1+p^2}}\text{arctanh} \left(\frac {1} {\sqrt{1+p^2}} \right)\right)\Theta(p-1)\label{ER}
\end{align}
and
\begin{align}
S(p)=\left(-\frac {2\log p} {1-p^2} +\frac {2} {p\sqrt{1-p^2}}\text{arctanh} \left(\frac {p} {\sqrt{1-p^2}} \right)\right)\Theta(1-p)+\nonumber\\
+\left(\frac {2\log p} {p(p^2-1)} +\frac {2} {p\sqrt{p^2-1}}\text{arctanh} \left(\frac {1} {\sqrt{p^2-1}} \right)\right)\Theta(p-1)\label{ES}
\end{align}
where $\Theta$ is the Heaviside's function. The function $\mathcal V\in C^\infty ((0, 1)\cup (1, \infty))$ is such that
\begin{align}
&\mathcal V(p)\sim -4p^2(\log p-\frac {2} {3}+\cdots),\,\,p\to 0 \label{E2}\\
&\mathcal V(p)\sim -\frac {4} {p^5}(\log p+\frac {2} {3}+\cdots),\,\,p\to \infty\label{E3}\\
&\mathcal V(p)\sim \frac {\pi } {\sqrt 2 \sqrt{|p-1|}}-1-\sqrt 2\, \text{arcoth}(\sqrt 2)+\cdots,\,p\to 1.\label{E3b}
\end{align}
The function $\mathcal V$  has a singularity at $p=1$ of order $|p-1|^{-1/2}$. Then  $\mathcal V\in L^p(0, 1)$ for $p\in [1, 2)$.
Moreover, we record the following behaviour of the first derivative
\begin{align}
&\mathcal V'(p)\sim -\frac {4} {3}(-1+6\log p )p+\mathcal O(p)^3,\,\,p\to 0 \label{EC2}\\
&\mathcal V'(p)\sim -\frac {4(7+15\log p)} {3p^6}+\mathcal O(p)^{-8},\,\,p\to \infty\label{EC3}\\
&\mathcal V'(p)\sim \frac {\pi } {2\sqrt 2\, (1-p)^{3/2}}-\frac {5 \pi } {8\sqrt 2\, \sqrt{1-p}}+\cdots ,\,p\to 1^-\label{EC3b}\\
&\mathcal V'(p)\sim -\frac {\pi } {2\sqrt 2\, (p-1)^{3/2}}-\frac {5 \pi } {8\sqrt 2\, \sqrt{p-1}}+\cdots ,\,p\to 1^+, \label{EC3b2}
\end{align}
while for the second derivative we have 
\begin{align}
&\mathcal V''(p)\sim -\frac {4} {3}(5+6\log p )+\mathcal O(p),\,\,p\to 0 \label{ED2}\\
&\mathcal V''(p)\sim -\frac {12(3+10\log p)} {p^7}+\mathcal O(p)^{-8},\,\,p\to \infty\label{ED3}\\
&\mathcal V''(p)\sim \frac {3\pi } {4\sqrt 2\, (1-p)^{5/2}}+\frac {5 \pi } {16\sqrt 2\, (1-p)^{3/2}}+\cdots ,\,p\to 1^-\label{ED3b}\\
&\mathcal V''(p)\sim \frac {3\pi } {4\sqrt 2\, (p-1)^{5/2}}+\frac {5 \pi } {16\sqrt 2\,  (p-1)^{3/2}}+\cdots ,\,p\to 1^+.\label{ED3b2}
\end{align}

The Mellin transform of $\mathcal V$ was also obtained in \cite{MM}. It is a meromorphic function $W _{ \mathcal V }$ such that $\gamma +W _{ \mathcal V }$ is analytic on the domain $D$ where $\Re e(s)\in (-2, 5)$, has two double poles at $s=-2$ and $s=5$ and  
\begin{align}
W(s)=\frac {4} {(s+2)^2}+\frac {8} {3(s+2)}+\mathcal O(1),\,s\to -2,\,s\in D \label{ED2W}\\
W(s)=\frac {4} {(s-5)^2}-\frac {8} {3(s-5)}+\mathcal O(1),\,s\to 5,\,s\in D. \label{ED3W}
\end{align}
Moreover, 
\begin{align}
W'(s)=-\frac {8} {(s+2)^3}-\frac {8} {3(s+2)^2}+\mathcal O(1),\,s\to -2,\,s\in D \label{EDP2W}\\
W'(s)=-\frac {8} {(s-5)^3}+\frac {8} {3(s-5)^2}+\mathcal O(1),\,s\to 5,\,s\in D, \label{EDP3W}
\end{align}
while
\begin{align}
W''(s)=\frac {24} {(s+2)^4}+\frac {16} {3(s+2)^3}+\mathcal O(1),\,s\to -2,\,s\in D \label{EDpp2W}\\
W''(s)=\frac {24} {(s-5)^3}-\frac {16} {3(s-5)^2}+\mathcal O(1),\,s\to 5,\,s\in D.\label{EDpp3W}
\end{align}

By definition of $K$ and  properties (\ref{E2})-(\ref{E3b}), 
\begin{align}
K(x)&=\mathcal V(p)\sim  -4p^2(\log p-\frac {2} {3}+\cdots)=-4e^{x}\left(\frac {x} {2}-\frac {2} {3}+\cdots\right),\,\,x\to -\infty \label{E13}\\
K(x)&=\mathcal V(p)\sim -\frac {4} {p^5}\left(\log p+\frac {2} {3}+\cdots \right)=4e^{-\frac {5x} {2}}\left( \frac {x} {2}+\frac {2} {3}+\cdots \right),\,\,x\to \infty,  \label{E14}\\
K(x)&=\mathcal V(p)\sim  \frac {\pi } {\sqrt 2 \sqrt{|e^{\frac {x} {2}}-1|}}-1-\sqrt 2\, \text{arcoth}(\sqrt 2)+\cdots 
= \frac {\pi } {\sqrt {|x|}}-\frac {\pi \sqrt{|x|}} {8}+\cdots,\,x\to 0. \label{E14b}
\end{align}
Since,
\begin{align*}
K'(x)=\mathcal V'(p)\frac {\dd p} {\dd x}=\mathcal V'(p)\frac {p} {2}
\end{align*}
it follows from (\ref{EC2})-(\ref{EC3b2}),
\begin{align}
K'(x )&\sim -\frac {4} {6}(-1+3x)e^{x} +\mathcal O(e^{\frac {x} {2}})^4,\,x\to -\infty,  \label{EC2K}\\
K'(x )&\sim -\frac {4} {6}\left(7+\frac {15} {2}x\right)e^{-\frac {5x} {2}} +\mathcal O(e^{\frac {x} {2}})^{-7},\,x\to \infty,   \label{EC3K}\\
K'(x)&\sim  \frac {\pi e^{\frac {x} {2}} } {4\sqrt 2\, (1-e^{\frac {x} {2}} )^{3/2}}-\frac {5 \pi e^{\frac {x} {2}}  } {16\sqrt 2\, \sqrt{1-e^{\frac {x} {2}} }}+\cdots ,\,x\to 0^-\nonumber\\
&\sim \frac {\pi } {2|x|^{3/2}}-\frac {3\pi } {8 |x|^{1/2}}+\mathcal O(|x|)^{1/2},\,x\to 0^- \label{EC3bK}\\
K'(x) &\sim 
-\frac {\pi e^{\frac {x} {2}} } {4\sqrt 2\, (e^{\frac {x} {2}} -1)^{3/2}}-\frac {5 \pi  e^{\frac {x} {2}} } {16\sqrt 2\, \sqrt{e^{\frac {x} {2}} -1}}+\cdots 
,\,x\to 0^+ 
\nonumber\\
&
\sim -\frac {\pi } {2|x|^{3/2}}-\frac {3\pi } {8 |x|^{1/2}}+\mathcal O(|x|)^{1/2},\,x\to 0^+. \label{EC3b2K}
\end{align}
Similarly, from \eqref{ED2} - \eqref{ED3b2}, 
and
$K''(x)=\mathcal V''(p)\left(\frac {\dd p} {dx}\right)^2+\mathcal V'(p)\frac{ \dd^2p}{\dd x^2}$, $\dd p/\dd x=e^{x/2}/2$,
\begin{align}
&|K''(x)|\le C|x|e^{x},\,\,x \to -\infty \label{ED2K}\\
&|K''(x)|\le C|x|e^{-\frac {5x} {2}},\,\,x \to \infty \label{ED3K}\\
&K''(x)\sim  \frac {3\pi } {4 |x|^{5/2} }+\mathcal O\left( |x|^{3/2}\right),\,x\to 0.
\end{align}

\begin{prop} It holds that 
\begin{equation}
\label{KastK0}
K\ast K(x)=
\begin{cases}
\frac {8x^3} {3}e^{-\frac {5 x} {2}}+\mathcal O(xe^{-\frac {5 x} {2}}),\,x\to \infty\\
-\frac {8x^3} {3}e^{ x}+\mathcal O(xe^{x}),\,x\to -\infty.
\end{cases}
\end{equation}
\end{prop}
\begin{proof}
Suppose $x\to \infty$ and write,
\begin{align}
\label{KastK1}
K\ast K(x)=\int  _{ -\infty }^0K(x-y)K(y)\dd y+\int_0^xK(x-y)K(y)\dd y+\nonumber\\+\int  _x^\infty K(x-y)K(y)\dd y.
\end{align}
By (\ref{E14}),
\begin{align*}
&\int_x^\infty K(x-y)K(y)\dd y\sim 
4 \int _x^\infty ye^{-\frac {5y} {2}}K(x-y) \dd y\\
&=4 \int _{ -\infty }^0 (x-z)e^{-\frac {5(x-z)} {2}}K(z)\dd z \\ & 
=
4 xe^{-\frac {5x} {2}}\int _{ -\infty }^0  e^{\frac {5z} {2}}K(z) \dd z
 -4 e^{-\frac {5x} {2}}\int _{ -\infty }^0  ze^{\frac {5z} {2}}K(z) \dd z.
\end{align*}
where, by (\ref{E13}),
\begin{align*}
\int _{ -\infty }^0  e^{\frac {5z} {2}}K(z)\dd z<\infty \,\,\text{ and}\,\,\,\int _{ -\infty }^0  ze^{\frac {5z} {2}}K(z) \dd z<\infty.
\end{align*}
This implies
\begin{align}
\label{KastK2}
\int_x^\infty K(x-y)K(y)\dd y=
4xe^{-\frac {5x} {2}}\int _{ -\infty }^0  e^{\frac {5z} {2}}K(z)\dd z+
\mathcal O\left(e^{-\frac {5x} {2}} \right),\,x\to \infty.
\end{align}
On the other hand, from \eqref{E13}, \eqref{E14}
\begin{align}
\label{KastK6}
&\int_{ -\infty }^0 K(x-y)K(y)\dd y
\le C\int_{ -\infty}^0 (x-y)e^{-\frac {5(x-y)} {2}}ye^{y}\dd y\nonumber\\
&\le C xe^{-\frac {5x} {2}}\int_{ -\infty }^0 e^{\frac {7y} {2}}y \dd y+
C e^{-\frac {5x} {2}}\int_{ -\infty }^0 e^{\frac {7y} {2}}y^2 \dd y=
\mathcal O\left(xe^{-\frac {5x} {2}} \right).
\end{align}
The second integral in the right hand side of (\ref{KastK1}) gives first,
\begin{align*}
\int _0^xK(x-y)K(y)\dd y=2\int _0^{x/2}K(x-y)K(y)\dd y.
\end{align*}
As $x\to \infty$, $x-y>x/2\to \infty$ and so by \eqref{E14},
\begin{align*}
\int _0^{x/2}K(x-y)K(y)\dd y\sim 4\int _0^{x/2}(x-y)e^{-\frac {5(x-y)} {2}}K(y)\dd y\\
=4xe^{-\frac {5x} {2}}\int _0^{x/2} e^{\frac {5y} {2}}K(y)\dd y-4e^{-\frac {5x} {2}}\int _0^{x/2}ye^{\frac {5y} {2}}K(y)\dd y.
\end{align*}
For $R\in (0, x/2)$ large ,
\begin{align*}
\int _0^{x/2} e^{\frac {5y} {2}}K(y)\dd y=\int _0^{R} e^{\frac {5y} {2}}K(y)\dd y+\int _R^{x/2} e^{\frac {5y} {2}}K(y)\dd y
\end{align*}
For $y>R$, $K(y)=ye^{-\frac {5y} {2}}+\mathcal O\left(e^{-\frac {5y} {2}}\right)$, then
\begin{align*}
&\int_R^{x/2} e^{\frac {5y} {2}}K(y)\dd y=\int _R^{x/2} e^{\frac {5y} {2}} \left(4ye^{-\frac {5y} {2}}+\mathcal O\left(e^{-\frac {5y} {2}}\right) \right)\dd y\\
&=4\int _R^{x/2} y \dd y+\int _R^{x/2} e^{\frac {5y} {2}}\mathcal O\left(e^{-\frac {5y} {2}}\right) \dd y
=2\left(\frac {x^2} {4}-R^2 \right)+\mathcal O\left(x \right)\\
&= \frac {x^2} {2}+\mathcal O(R^2)+\mathcal O\left(x \right),\,x\to \infty,
\end{align*}
and then, 
\begin{align*}
\int _0^{x/2} e^{\frac {5y} {2}}K(y)\dd y= \frac {x^2} {2}+\mathcal O(R^2)+\mathcal O\left(x \right),\,x\to \infty.
\end{align*}
A similar argument in the integral
\begin{align*}
\int _0^{x/2} e^{\frac {5y} {2}}yK(y)\dd y=\int _0^{R} ye^{\frac {5y} {2}}K(y)\dd y+\int _R^{x/2}y e^{\frac {5y} {2}}K(y)\dd y
\end{align*}
yields,
\begin{align*}
&\int _R^{x/2} ye^{\frac {5y} {2}}K(y)\dd y=\int _R^{x/2} e^{\frac {5y} {2}} y\left(4ye^{-\frac {5y} {2}}+\mathcal O\left(e^{-\frac {5y} {2}}\right) \right)dy\\
&=4\int _R^{x/2} y^2 \dd y+\int _R^{x/2} e^{\frac {5y} {2}}y\mathcal O\left(e^{-\frac {5y} {2}}\right) dy
=\frac {4} {3}\left(\frac {x^3} {8}-R^3 \right)+\mathcal O\left(x^2 \right)\\
&= \frac {x^3} {6}+\mathcal O(R^3)+\mathcal O\left(x ^2\right),\,x\to \infty.
\end{align*}
Then, 
\begin{align*}
\int _0^{x/2} e^{\frac {5y} {2}}yK(y)dy= \frac {x^3} {6}+\mathcal O(R^3)+\mathcal O\left(x ^2\right),\,x\to \infty.
\end{align*}
It follows,
\begin{align*}
&\int _0^{x/2}K(x-y)K(y)dy=4xe^{-\frac {5x} {2}}\Bigg(  \frac {x^2} {2}+\mathcal O(R^2)+\mathcal O\left(x \right)\Bigg)-\\
&-4e^{-\frac {5x} {2}}\Bigg(\frac {x^3} {6}+\mathcal O(R^3)+\mathcal O\left(x ^2\right) \Bigg)
=\frac {4x^3} {3}e^{-\frac {5x} {2}}+\mathcal O\left(x^2 e^{-\frac {5x} {2}} \right),\,\,x\to \infty
\end{align*}
and
\begin{align}
\label{KastK3}
\int _0^xK(x-y)K(y)dy=\frac {8x^3} {3}e^{-\frac {5x} {2}}+\mathcal O\left(x^2 e^{-\frac {5x} {2}} \right),\,\,x\to \infty.
\end{align}
Property (\ref{KastK0}) for $x\to \infty$ follows from  (\ref{KastK2}), (\ref{KastK6}) and (\ref{KastK3}). A similar argument shows (\ref{KastK0}) for $x\to -\infty$.
\end{proof}

\textbf{Acknowledgments.}
The research of the first author is supported by grant PID2020-112617GB-C21 of MCIN, grant  IT1247-19 of the Basque Government and grant RED2022-134784-T funded by MCIN/AEI/10.13039/501100011033. The second author acknowledges support from the Chapman Fellowship at Imperial College London.

\end{document}